\documentclass[a4paper,reqno]{amsart}

\usepackage{enumerate}
\usepackage[usenames]{color}

\usepackage[colorlinks=true]{hyperref}
\hypersetup{urlcolor=blue, linkcolor=blue, citecolor=red}

\numberwithin{equation}{section}

\newtheorem{thm}{Theorem}[section]
\newtheorem{lem}[thm]{Lemma}
\newtheorem{prop}[thm]{Proposition}

\newtheorem{defn}[thm]{Definition}

\theoremstyle{definition}
\newtheorem{rem}[thm]{Remark}

\newcommand\R{{\mathbb R}}
\newcommand\C{{\mathbb C}}

\newcommand\Tma{T_{\mathrm{max}}}

\newcommand\Comp{{\mathrm{c}}}

\newcommand\DI{\varphi  }
\newcommand\DIb{\psi }

\newcommand\dist{{\mathrm{d}}}
\newcommand\Yset{{\mathcal Y}}
\newcommand\Admis{{\mathcal A}}
\newcommand\CSob{A}
\newcommand\CStr{K}
\newcommand\CSTu{C_1}
\newcommand\CSTub{C_2}

\newcommand\CHo{B}
\newcommand\CHov{\Gamma }
\newcommand\Rhostar{\nu }
\newcommand\Twostar{\mu }
\newcommand\Betastar{\theta }

\newcommand\Step[1]{\medskip \noindent {\sc Step~#1.}\quad}
\newcommand\goto{\mathop{\longrightarrow}}

\newcommand\Srn{{\mathcal S}(\R^N )}
\newcommand\union{\mathop{\cup}}

\newcommand\MScN[1]{\href{http://www.ams.org/mathscinet-getitem?mr=#1}{\nolinkurl{(#1)}}}
\newcommand\DOI[1]{\href{http://dx.doi.org/#1}{(doi: \nolinkurl{#1})}}
\newcommand\LINK[1]{\href{#1}{(link: \nolinkurl{#1})}}

\title[the $H^2$-critical nonlinear Schr\"o\-din\-ger equation]{Local well-posedness for the $H^2$-critical nonlinear Schr\"o\-din\-ger equation}

\author[Thierry Cazenave]{Thierry Cazenave$^1$}
\author[Daoyuan Fang]{Daoyuan Fang$^{\dag,2}$}
\author[Zheng Han]{Zheng Han$^{\dag,2}$}

\address{$^1$Universit\'e Pierre et Marie Curie \& CNRS, Laboratoire Jacques-Louis Lions,
B.C. 187, 4 place Jussieu, 75252 Paris Cedex 05, France}

\address{$^{2}$Department of Mathematics, Zhejiang University, Hangzhou, 310027,
China}

\email[Thierry Cazenave]{\href{mailto:thierry.cazenave@upmc.fr}{thierry.cazenave@upmc.fr}}
\urladdr[Thierry Cazenave]{\href{http://www.ljll.math.upmc.fr/~cazenave/}{http://www.ljll.math.upmc.fr/~cazenave/}}
\email[Daoyuan Fang]{\href{mailto:dyf@zju.edu.cn}{dyf@zju.edu.cn}}
\email[Zheng Han]{\href{mailto:hanzheng5400@yahoo.com.cn}{hanzheng5400@yahoo.com.cn}}

\thanks{$^\dag$Research supported by grants 11271322 and 10931007 of the NSFC, China}



\begin{document}

\maketitle

\begin{abstract}
In this paper, we consider the nonlinear Schr\"o\-din\-ger equation $iu_t +\Delta u= \lambda  |u|^{\frac {4} {N-4}} u$ in $\R^N $, $N\ge 5$, with $\lambda \in \C$. 
We prove local well-posedness (local existence, unconditional uniqueness, continuous dependence) in the critical space $\dot H^2 (\R^N ) $. 
\end{abstract}

\section{Introduction}

Throughout this paper, we assume  $N\ge 5$ and
consider the $H^2$-critical nonlinear Schr\"o\-din\-ger equation 
\begin{equation} \label{NLS} \tag{NLS} 
\begin{cases} 
iu_t + \Delta u= \lambda  |u|^\alpha u,\\
u(0)= \DI,
\end{cases} 
\end{equation} 
in $\R^N $, where $\lambda \in \C$ and
\begin{equation} \label{fAl1} 
\alpha = \frac {4} {N-4}.
\end{equation} 
It is often convenient to study the equivalent form equation~\eqref{NLS}
\begin{equation} \label{NLSI} 
u(t) = e^{it \Delta }\DI - i \lambda \int _0^t e^{i(t-s) \Delta }   |u (s) |^\alpha u(s) \,ds,
\end{equation} 
where $(e^{it \Delta }) _{ t\in \R }$ is the Schr\"o\-din\-ger group. (See, e.g., Lemma~1.1 in~\cite{Kato2}.)

Local existence for the Cauchy problem~\eqref{NLS} is well known in the Sobolev space $H^s (\R^N ) $ provided $\alpha <  \frac {4} {N-2s}$ and (if $s>1$) that the nonlinearity is sufficiently smooth.  See Kato~\cite{Kato2}, Tsutsumi~\cite{Tsutsumi},
Cazenave and Weissler~\cite{CazenaveW}, Kato~\cite{Kato1}.
The smoothness condition on the nonlinearity can be improved (removed, if $s\le 2$) by estimating time derivatives of the solution instead of space derivatives. 
See Kato~\cite{Kato3}, Pecher~\cite{Pecher}, Fang and Han~\cite{FangH}. 
The solution depends continuously on the initial value $H^s \to C([0,T], H^s)$, see Kato~\cite{Kato2}, Tsutsumi~\cite{Tsutsumi}, Cazenave, Fang and Han~\cite{CazenaveFH},
Dai, Yang  and Cao~\cite{DaiYC}, Fang and Han~\cite{FangH}. 
Unconditional uniqueness (i.e., uniqueness in $C([0,T], H^s)$ or $L^\infty ((0,T), H^s)$, without assuming the solution belongs to some auxiliary space) is known in a number of cases, see Kato~\cite{Kato1}, Furioli and Terraneo~\cite{FurioliT},  Rogers~\cite{Rogers}, 
Fang and Han~\cite{FangH2}. Many of these results hold in the critical case $\alpha =\frac {4} {N-2s}$, see Cazenave and Weissler~\cite{CazenaveW}, Kato~\cite{Kato1}, Cazenave~\cite{CLN},  Kenig and Merle~\cite{KenigM},  Tao and Visan~\cite{TaoV}, Killip and Visan~\cite{KillipV}, 
Cazenave, Fang and Han~\cite{CazenaveFH}, Win and Tsutsumi~\cite{WinT}, Fang and Han~\cite{FangH2}.

Our main result concerns the  $H^2$-critical case~\eqref{fAl1}, and is the following. 
(The homogeneous Sobolev space $\dot H^2 (\R^N )$ as well as the admissible pairs are defined in Section~\ref{sPrelim} below.)

\begin{thm} \label{eComplete} 
Suppose $N\ge 5$, $\lambda \in \C$ and $\alpha $ is given by~\eqref{fAl1}.
Given any $\DI \in  \dot H^2 (\R^N ) $, there exist a maximal existence time $\Tma = \Tma (\DI )>0$ and a unique solution $u\in  C([0, \Tma ), \dot H^2 (\R^N ) ) $ of~\eqref{NLS}.
If, in addition, $\DI\in L^2 (\R^N ) $  then $u\in  C([0, \Tma ),  H^2 (\R^N ) ) $.
Moreover, the following properties hold.
\begin{enumerate}[{\rm (i)}] 

\item \label{eComplete:2} 
$\Delta u \in L^q ((0,T), L^r (\R^N )) $ and $u_t\in   L^q ((0,T), L^r (\R^N ) ) \cap C ([0, T] ), L^2 (\R^N ) )$ for every $T<\Tma$ and every admissible pair $(q,r)$.

\item \label{eComplete:1} 
$u - e^{i \cdot \Delta }\DI \in   L^q ((0,T), L^r (\R^N ) ) \cap C ([0,  T], L^2 (\R^N ) )$ for every $T<\Tma$ and every admissible pair $(q,r)$.

\item \label{eComplete:3} If $ \| \Delta \DI \| _{ L^2 }$ is sufficiently small, then $\Tma =\infty $ and 
both $u_t$ and $\Delta u$ belong to $L^q ((0,\infty ), L^r (\R^N ) )$ for every admissible pair $(q,r)$. If, in addition, $\DI \in L^2 (\R^N ) $, then also $u\in L^q ((0,\infty ), L^r (\R^N ) )$. 

\item \label{eComplete:4} (Blowup alternative.)
If $\Tma <\infty $, then $ \| u \| _{ L^\gamma ((0, \Tma), L^\Rhostar  )}  = \infty $, where $\gamma = \frac {2(N-2)} {N-4}, \Rhostar=  \frac {2N (N-2)} {(N-4)^2}$.

\item \label{eComplete:5} (Continuous dependence.)
Let $\DI \in  \dot H^2 (\R^N ) $, $(\DI^n) _{ n\ge 1 } \subset \dot H^2 (\R^N )$, let $u$ and $(u^n) _{ n\ge 1 }$ be the corresponding solutions of~\eqref{NLS} and let $\Tma $ and $(\Tma^n) _{ n\ge 1 }$ denote their respective maximal existence times. 
Suppose $\DI^n \to \DI$ in $\dot H^2 (\R^N ) $ as $n\to \infty $. If $0<T< \Tma$, then $\Tma^n >T$ for all sufficiently large $n$. Moreover, $\Delta u^n \to \Delta u$  and $u^n_t \to u_t$ in $L^q ((0,T), L^r (\R^N )) $ as $n\to \infty $, for every admissible pair $(q,r)$. 
If, in addition, $\DI\in L^2 (\R^N ) $, $(\DI^n) _{ n\ge 1 }\subset L^2 (\R^N ) $ and $\DI^n \to \DI$ in $L^2 (\R^N ) $, then $u^n \to u$ in $L^q ((0,T), L^r (\R^N )) $. 

\end{enumerate} 
\end{thm}

Note that if $u\in C([0, T], \dot H^2 (\R^N ) )$, then $ |u|^\alpha u\in C([0, T], L^2 (\R^N ) )$ by Sobolev's embedding (see~\eqref{fAl17} below), so that  equation~\eqref{NLS} makes sense in $L^2 (\R^N ) $. 

We note that Theorem~\ref{eComplete} is the $H^2$ counterpart of what is known in the $H^1$-critical case $\alpha =\frac {4} {N-2}$, $N\ge 3$. 
The existence part of Theorem~\ref{eComplete} in the nonhomogeneous space $H^2 (\R^N )$ is well-known  for  $N\le 7$  \cite[Theorem~1.2]{CazenaveW}, and for $N\ge 8$ and
small initial values $\DI$ \cite[Theorem~1.4]{CazenaveW}.
Local existence for large data in $H^2$ when $N\ge 8$, local existence and unconditional uniqueness in the homogeneous space $\dot H^2$, and continuous dependence are new, as far as we are aware. 
Our proof of the existence part of Theorem~\ref{eComplete} follows essentially the proof in~\cite{CazenaveW}, which is a fixed-point argument. 
The only noticeable modifications are Lemmas~\ref{eNLem1} and~\ref{eNLem2}  below, which  provide estimates of the nonlinear term $ |u|^\alpha u$. These estimates replace, in our proof, the estimates given by Lemma~5.6 in~\cite{CazenaveW}, and allow us to remove the small data requirement in~\cite[Theorem~1.4]{CazenaveW}. Moreover, the set in which the fixed-point is constructed is modified, with respect to~\cite{CazenaveW}, in order to consider initial values in the homogeneous space $\dot H^2$.  (See Definition~\ref{eSetY} below.) This modification is also the key to prove property~\eqref{eComplete:1} of Theorem~\ref{eComplete}, which means that the nonlinear term in~\eqref{NLSI} has better regularity properties than the solution $u$ itself. 
Unconditional uniqueness follows from the argument used in~\cite[Proposition~4.2.5]{CLN} for the $H^1$-critical case. Continuous dependence is established by adapting the method of~\cite[Theorem~III$'$]{Kato3}. Since we are in the critical case, a truncation argument is used, as in the proof of unconditional uniqueness. 

The rest of this paper is organized as follows. In Section~\ref{sPrelim}, we introduce the notation and establish a few useful estimates. In Section~\ref{sUU}, we prove unconditional uniqueness. Section~\ref{sLocal} is devoted to local existence and Section~\ref{sCDep} to local continuous dependence. We complete the proof of Theorem~\ref{eComplete} in Section~\ref{sProofMain}. 
Finally,  an appendix is devoted to the proof of a technical lemma (Lemma~\ref{eElem2v2} below).

\section{Notation and preliminary results} \label{sPrelim} 
Throughout this paper, all the  function spaces we consider are made up of complex-valued functions. 
Given $1\le p\le \infty $, we denote by $p'$ the conjugate exponent defined by $\frac {1} {p'}= 1- \frac {1} {p}$. 
We say that a pair $(q,r)$ is admissible if $(q,r)\in \Admis$, where
\begin{equation} \label{fAdmis} 
\Admis = \Bigl\{ (q,r)\in [2,\infty ]\times [2, 2N/(N-2)];\, \frac {2} {q}= N \Bigl( \frac {1} {2}- \frac {1} {r} \Bigr) \Bigr\}, 
\end{equation} 
and we recall Strichartz's estimates
\begin{gather} 
\sup  _{ (q,r) \in \Admis} \| e^{i \cdot \Delta } w\| _{ L^q (\R , L^r )  } \le  \CStr   \| w \| _{ L^2 }, \label{fAl25} \\
\sup  _{ (q,r) \in \Admis}  \Bigl\| \int _0^\cdot  e^{i(\cdot -s) \Delta }f(s)\,ds \Bigr\| _{ L^q ((0,T), L^r )}\le   \CStr  \inf _{  (q,r) \in \Admis } \|f\| _{  L^{q '} ((0,T), L^{r '}) },  \label{fAl26}
\end{gather} 
valid for $0<T\le \infty $, 
where the constant $ \CStr $ depends only on $N$. 
Moreover, if $w\in L^2 (\R^N ) $, then $e^{i \cdot \Delta } w \in C(\R, L^2 (\R^N ) )$; and if the right-hand side of~\eqref{fAl26} is finite, then the integral on the left-hand side belongs to $C([0,T], L^2 (\R^N ) )$ if $T<\infty $ and to $C([0,\infty ), L^2 (\R^N ) )$ if $T=\infty $.
(See~\cite{Strichartz, KeelT}.)

It is convenient to introduce the numbers
\begin{gather} 
\rho = \frac {2N (N-2)} {N^2 - 4N +8},\quad \gamma = \frac {2(N-2)} {N-4}, \label{fAl2} \\
\beta  = \frac {2N^2} {N^2-2N+8}, \quad \Twostar = \frac {2N} {N-4}, \label{fAl3b1}
\end{gather} 
and
\begin{equation} \label{fAl4} 
 \Rhostar = \frac {2N (N-2)} {(N-4)^2}, \quad \Betastar = \frac {2N^2} {(N-2)(N-4)}.
\end{equation} 
It is straightforward to verify that $(\gamma ,\rho )$ and $(\Twostar, \beta )$ are admissible pairs and that $2 < \beta  < \rho $.
We also recall the following
Sobolev's inequalities (see~\cite{BerghL, Triebel}). 
\begin{equation} \label{fAl17}
 \| u \| _{ L^{ \Twostar }  }\le  \CSob   \| \Delta u \| _{ L^2 },\quad   \| u \| _{ L^{ \Rhostar } } \le  \CSob   \| \Delta u\| _{ L^\rho  }, \quad \| u \| _{ L^{\Betastar} } \le \CSob  \| \Delta u\| _{ L^{\beta } }.
\end{equation} 
We consider the space $\dot H^2 (\R^N ) $
defined as the completion of $\Srn$ for the norm $ \| u \| _{ \dot H^2 }=  \| \Delta u \| _{ L^2 }$. 
Alternatively, in view of~\eqref{fAl17}, $\dot H^2 (\R^N ) $ is the set of $u\in L^\Twostar (\R^N )$ such that $\Delta u\in L^2 (\R^N ) $. 
Similarly, $\dot H^{2,\rho } (\R^N ) $ is the completion of $\Srn$ for the norm $ \| u \| _{ \dot H^{2, \rho } }=  \| \Delta u \| _{ L^\rho  }$ or, equivalently,  the set of $u\in L^\Rhostar (\R^N )$ such that $\Delta u\in L^\rho  (\R^N ) $. Note in particular that by~\eqref{fAl17}, $\dot H^2 (\R^N ) \hookrightarrow L^\Twostar (\R^N ) $ and $\dot H^{2, \rho }(\R^N )\hookrightarrow L^\Rhostar (\R^N ) $. As is well known, the Schr\"o\-din\-ger group $(e^{it\Delta }) _{ t\in \R }$ is a group of isometries on $L^2 (\R^N ) $ and on $\dot H^2 (\R^N ) $.

We now introduce the set $\Yset _{ \DI, T, M }$, in which we construct local solutions of~\eqref{NLS} by a fixed-point argument (see Section~\ref{sLocal}).

\begin{defn} \label{eSetY} 
Let $T, M>0$ and $\DI \in  \dot H^2 (\R^N ))$. 
We denote by $\Yset _{ \DI, T, M }$ the set of $u$ such that
\begin{enumerate}[{\rm (i)}] 

\item  \label{eSetY:1} $u\in  L^\gamma ((0,T),  \dot H^{2, \rho } (\R^N ))$ and $ \| \Delta u\| _{ L^\gamma ((0,T), L^\rho ) }\le M$,

\item  \label{eSetY:2}  $u_t \in L^\gamma ((0,T),  L^\rho (\R^N ))$ and $ \|u_t\| _{ L^\gamma ((0,T), L^\rho ) } \le M$,

\item  \label{eSetY:3}  $u- e^{i \cdot  \Delta }\DI\in  L^\gamma ((0,T),  L^\rho (\R^N ))$,

\item  \label{eSetY:4}  $u(0) = \DI $,
\end{enumerate} 
\noindent where $\rho$ and $\gamma $ are given by~\eqref{fAl2}. 
Moreover, we set
\begin{equation*} 
\dist (u,v)=  \|u-v\| _{ L^\gamma ((0,T), L^\rho ) },
\end{equation*}
for all  $u,v\in \Yset_{ \DI, T,M }$.
\end{defn} 

\begin{rem} \label{eRem1} 
\begin{enumerate}[{\rm (i)}] 

\item \label{eRem1:1} Note that by Strichartz's estimate~\eqref{fAl25}, $\partial _t e^{i\cdot \Delta }\DI = ie^{i \cdot \Delta } \Delta \DI \in L^\gamma (\R , L^\rho (\R^N ) )$. 
Thus if $u\in \Yset _{ \DI, T, M } $, then $u- e^{i \cdot  \Delta }\DI\in  W^{1, \gamma } ((0,T),  L^\rho (\R^N ))  $, so that $u - e^{i \cdot  \Delta }\DI \in C([0,T), L^\rho (\R^N ) )$. Since $e^{i \cdot  \Delta }\DI\in  C(\R, \dot H^2 (\R^N ) ) \hookrightarrow C(\R, L^\Twostar (\R^N ) )$, we see that $u\in C([0,T) , L^\rho  (\R^N ) +L^\Twostar (\R^N ) )$. In particular, the condition $u(0)= \DI $ in Definition~$\ref{eSetY}$  makes sense in $L^\rho  (\R^N ) +L^\Twostar (\R^N )$. 

\item \label{eRem1:2} It is clear that $\dist$ is a distance on $\Yset _{ \DI, T, M }  $ and it is
not difficult to show that $(\Yset _{ \DI, T, M }  , \dist)$ is a complete metric space. 

\end{enumerate} 
\end{rem} 

In the rest of this section, we establish  useful estimates of functions in $\Yset _{ \DI, T, M } $.
To prove these estimates, we will use the following elementary inequalities.

\begin{lem} \label{eElem} 
Given any $a >0$, there exists a constant $ C(a )$ such that
\begin{gather} 
  | \,  |u|^a u-  |v|^a v |\le (a +1) ( |u| + |v|)^a  |u-v|, \label{fSPL1}  \\
 |\,  |u|^a  u - |v|^a  v| \le C(a )  |\,  |u|^{a  +1} u - |v|^{a  +1} v|^{\frac {a  +1} {a  +2}}, \label{fElem1} 
\end{gather} 
and 
\begin{multline} \label{feRemBSt}
 |\,  |u|^a  - |v|^a   |+  |\,  |u|^{a  -2} u^2 -  |v|^{a  -2} v^2 | \\ \le
  \begin{cases}
  C(a ) ( |u|^{a  -1} +  |v|^{a  -1})   |u-v| & \text{if }a  \ge 1, \\
C (a )  |u-v|^a  &  \text{if }0<a  \le 1,
 \end{cases}
\end{multline}
for all $u,v\in \C$. 
\end{lem} 

\begin{proof} 
Estimate~\eqref{fSPL1} is immediate and~\eqref{feRemBSt} follows from~\cite[(2.26) and~(2.27)]{CazenaveFH}. 
We prove \eqref{fElem1} for completeness.  
Let $z\in \C$, $ |z|\le 1$. It follows that $ |z|\le  |z|^{\frac {1} {a +2}}$ so that
$  |1- |z| ^{\frac {1} {a +2}}| = 1- |z| ^{\frac {1} {a +2}} \le 1- |z| =  |1- |z|\, |\le  |1-z|$;
and so,
\begin{equation}  \label{fElem2} 
 |1-  |z|^{-\frac {1} {a +2}} z |  \le  |1-z|+  |z| ^{\frac {a +1} {a +2}}  |  |z|^{\frac {1} {a +2}} -1|   \le 2 |1-z|.
\end{equation} 
Since $ |1-z|\le 2$, we have $ |1-z|\le 2^{\frac {1} {a +2}}  |1-z|^{\frac {a +1} {a +2}}$ and we deduce from~\eqref{fElem2}  that
\begin{equation} \label{fElem3} 
|1-  |z|^{-\frac {1} {a +2}} z | \le 2^{\frac {a +3} {a +2}}  |1-z|^{\frac {a +1} {a +2}} \quad  \text{if }\quad  |z|\le 1. 
\end{equation} 
Let now  $u,v\in \C$ with $ |v| \le   |u|$ and $ |u|\not = 0$. Inequality~\eqref{fElem1} (with $C(a )= 2^{\frac {a +3} {a +2}} $) follows by setting $z=  |v/u|^{a +1}  (v/u)$ in~\eqref{fElem3} and multiplying by $ |u|^{a +1}$.
\end{proof} 

\begin{lem}\label{eElem2v2}
Let $T>0$, $a>0$. Let $q_1, q_2,r_1, r_2\ge 1$ satisfy $q_1, r_1\ge a+1$, $\frac {a} {q_1}+ \frac {1} {q_2}\le 1$, $\frac {a} {r_1}+ \frac {1} {r_2}\le 1$. If $u\in L^{q_1}  ( (0,T), L^{r_1} (\R^N )  )$ and $u_t \in  L^{q_2}((0,T), L^{r_2} (\R^N ) )$, then
\begin{equation} \label{eElem2v2:1}
 \partial _t ( |u|^a u)= \frac {a +2} {2}  |u|^a u _t + \frac {a} {2}  |u|^{a -2}u^2   \overline{u} _t,  
\end{equation} 
a.e. on $(0,T) \times \R^N $.
\end{lem}

The proof of Lemma~\ref{eElem2v2}, which uses an appropriate regularization argument, is postponed to the Appendix. 

\begin{lem} Given any $T>0$,
\begin{gather} 
 \| u \| _{ L^\Twostar ((0, T), L^\beta )} \le    \| u\| _{ L^\infty  ((0, T), L^2 )} ^{\frac {2} {N}}  \| u \| _{ L^\gamma  ((0, T), L^\rho  )} ^{\frac {N-2} {N}},\label{fHLD1}  \\
 \| u \| _{ L^{(\alpha +1)\gamma }((0,T) , L^{(\alpha +1) \rho })}^{\alpha +1}  \le 
 \| u \| _{ L^{\infty  }((0,T) , L^{ \Twostar })}^\alpha \| u \| _{ L^{ \gamma }((0,T) , L^{  \Rhostar  })} ,
 \label{fAl16}
\end{gather} 
and
\begin{multline} \label{fSP11} 
 \| \,  |u|^\alpha u-  |v|^\alpha v \| _{ L^{\gamma '}((0,T), L^{\rho '}) } \\ \le (\alpha +1) (
  \|u\| _{ L^\gamma ((0,T), L^{ \Rhostar } )} +   \|v\| _{ L^\gamma ((0,T), L^{ \Rhostar } )} )^\alpha 
  \|u- v\| _{ L^\gamma ((0,T), L^{\rho } )},
\end{multline}
hold for all functions $u,v$ for which the right-hand side makes sense. 
Moreover,
\begin{equation} \label{fAl23}
 \| \partial _t[  |u|^\alpha u]  \| _{ L^{\gamma '}((0,T), L^{\rho '} )}  \le (\alpha +1) 
  \|u\| _{ L^\gamma ((0,T), L^{ \Rhostar } )}^\alpha  \|u_t\| _{ L^\gamma ((0,T), L^\rho ) } ,
\end{equation} 
for all $u\in L^\gamma ((0,T), L^{ \Rhostar } (\R^N ))$ such that $u_t \in L^\gamma ((0,T), L^{ \rho  } (\R^N ))$.
\end{lem} 

\begin{proof} 
Both~\eqref{fHLD1} and~\eqref{fAl16} follow from H\"older's inequality in space and time,  by using the relations $\frac {1} {\beta }=  \frac {1} {N} +  \frac {N-2} {N\rho }$ for the first one, and
$\frac {1} {  \rho }=  \frac {\alpha } { \Twostar  } +  \frac {1} { \Rhostar }$ for the second one.
Estimates~\eqref{fSP11}  and~\eqref{fAl23}  follow from H\"older's inequality in space and time
and the relations $\frac {1} {\rho '}= \frac {\alpha } { \Rhostar }+ \frac {1} {\rho } $ and $\frac {1} {\gamma '}= \frac {\alpha } {\gamma }+ \frac {1} {\gamma } $, the first one by using~\eqref{fSPL1} and the  second one by using the inequality $ | \partial _t[  |u|^\alpha u]  | \le (\alpha +1)  |u|^\alpha  |u_t| $ (see~\eqref{eElem2v2:1}).
\end{proof} 

\begin{lem} 
Given any $T>0$, 
\begin{equation} \label{fFN1v2} 
 \| \,  |u|^\alpha u \| _{L^2 ((0,T) L^{\frac {2N} {N-2}} ) }   \le \CSob ^{\alpha +1}
   \| \Delta u\| _{ L^{\infty  } ((0,T),  L^{2 } )}^{\frac {2} {N-4}} 
  \| \Delta u\| _{ L^{\gamma } ((0,T),  L^{\rho } )}^{\frac {N-2} {N-4}} ,
\end{equation} 
and
\begin{multline} \label{fFN1v3} 
 \| \,  |u|^\alpha u - |v|^\alpha v\| _{L^2 ((0,T) L^{\frac {2N} {N-2}} ) }  \\  \le (\alpha +1) \CSob ^{\alpha +1}\CHov     \| \Delta (u-v) \| _{ L^{\infty  } ((0,T),  L^{2 } )}^{\frac {2} {N}} 
\| \Delta (u-v) \| _{ L^{\gamma } ((0,T),  L^{\rho } )}^{\frac {N-2} {N}} ,
\end{multline} 
where
\begin{equation*} 
\CHov =  \| \Delta u \| _{ L^{\infty  } ((0,T),  L^{2 } )}^\alpha + \| \Delta v \| _{ L^{\infty  } ((0,T),  L^{2 } )}^\alpha   + \| \Delta u \| _{ L^{\gamma  } ((0,T),  L^{\rho } )}^\alpha + \| \Delta v \| _{ L^{\gamma  } ((0,T),  L^{\rho } )}^\alpha ,
\end{equation*} 
hold for all $u,v\in L^\infty ((0,T), \dot H^2 (\R^N ) ) \cap L^\gamma  ((0,T), \dot H^{2, \rho } (\R^N ) ) $.
\end{lem} 

\begin{proof} 
Since 
\begin{equation*} 
 \| \  |u|^\alpha u\| _{L^2((0,T),  L^{\frac {2N} {N-2}}) } =  \| u \| _{L^\Twostar((0,T) , L^\Betastar ) }^{\alpha +1}
 \le  \CSob^{\alpha +1}  \| \Delta u \| _{L^\Twostar ((0,T),  L^\beta ) }^{\alpha +1},
\end{equation*} 
by~\eqref{fAl17}, estimate~\eqref{fFN1v2} follows from~\eqref{fHLD1} (applied with $u$ replaced by $\Delta u$). 
Similarly, we deduce from~\eqref{fSPL1} and~\eqref{fAl17} that
\begin{multline*} 
 \| \  |u|^\alpha u -  |v|^\alpha v \| _{L^2((0,T),  L^{\frac {2N} {N-2}}) } \\ \le  (\alpha +1)
  \CSob^{\alpha +1} [ \| \Delta u \| _{L^\Twostar ((0,T),  L^\beta ) }^{\alpha }
 +\| \Delta v \| _{L^\Twostar ((0,T),  L^\beta ) }^{\alpha }] \| \Delta (u-v) \| _{L^\Twostar ((0,T),  L^\beta ) } ,
\end{multline*} 
and~\eqref{fFN1v3} follows by applying~\eqref{fHLD1}.
\end{proof} 

\begin{lem} 
Given any $u , v\in \dot H^2 (\R^N ) $,
\begin{equation} \label{fAl24}
 \| \,  |u |^\alpha u \| _{ L^2 } \le  \CSob ^{\alpha +1}  \| \Delta u \| _{ L^2 }^{\alpha +1} ,
\end{equation} 
and 
\begin{equation} \label{fCs07} 
    \| \,  |u|^\alpha u -  |v|^\alpha v \| _{ L^2 } \le (\alpha +1) \CSob^{\alpha +1} ( \| \Delta u\| _{  L^2 }^\alpha +  \| \Delta  v\| _{ L^2 }^\alpha)  \| \Delta (u-v )\| _{ L^2 }.
\end{equation}
In particular, the map $u\mapsto  |u|^\alpha u$ is continuous $\dot H^2 (\R^N ) \to L^2 (\R^N ) $.
\end{lem} 

\begin{proof} 
Since  $\| \,  |u |^\alpha u \| _{ L^2 } =  \| u \| _{ L^{ \Twostar  } }^{\alpha +1} $, \eqref{fAl24}  follows from~\eqref{fAl17}. Similarly, we deduce from~\eqref{fSPL1} that
\begin{equation*} 
  \| \,  |u|^\alpha u -  |v|^\alpha v \| _{ L^2 } \le (\alpha +1) ( \| u\| _{ L^\Twostar }^\alpha +  \| v\| _{ L^\Twostar }^\alpha)  \| u-v \| _{ L^\Twostar },
\end{equation*} 
and~\eqref{fCs07} follows by applying~\eqref{fAl17}.
\end{proof} 

Next, given $\DI \in \dot H^2 (\R^N ) $, we set 
\begin{equation} \label{fAl33}
F(\DI, t)= \| e^{i \cdot \Delta } \Delta \DI\| _{ L^\gamma ((0,t), L^\rho ) } +  \| e^{i \cdot \Delta } [ |\DI|^\alpha \DI ] \| _{ L^\gamma ((0,t), L^\rho ) },
\end{equation} 
for  $0<t\le \infty $.  
We observe that, since $  | \DI|^\alpha \DI\in L^2 (\R^N ) $ by~\eqref{fAl24}, $F$ is well defined 
by~\eqref{fAl25}  and $F(\DI, t) \downarrow 0$ as $t\downarrow 0$. 
Moreover, it follows from~\eqref{fAl25} and~\eqref{fCs07} that the map $(\DI, t) \mapsto F(\DI, t)$  is continuous $\dot H^2 (\R^N ) \times (0, \infty ] \to (0,\infty )$. 
Therefore, if $E$ is a compact subset of $\dot H^2 (\R^N ) $, then
\begin{equation}  \label{fAl34}
\sup  _{ \DI\in E } F(\DI, t) \goto  _{ t\downarrow 0 }0. 
\end{equation} 
The next two lemmas are key ingredients in our proof of local existence and continuous dependence.

\begin{lem} \label{eNLem1} 
Let $T,M>0$, $\DI \in \dot H^2 (\R^N ) $, and let $\Yset _{ \DI, T, M } $ be as in Definition~$\ref{eSetY}$. If $u\in \Yset _{ \DI, T, M } $, then
\begin{multline}  \label{fAl49b1}
\| u \| _{ L^{(\alpha +1)\gamma } ((0,T), L^{(\alpha +1) \rho }) } ^{\alpha +1}  \le 
 \CSTu \| \Delta \DI \| _{ L^2 }^\alpha  F( \DI, T)  \\ 
  + \CSTu 
 ( \| u \| _{ L^\gamma ((0,T), L^\Rhostar )}^{\alpha +1}+ \| u_t \| _{ L^\gamma ((0,T), L^\rho )}^{\alpha +1} + F( \DI, T)^{\alpha +1} ), 
\end{multline} 
where $F $ is defined by~\eqref{fAl33} and the constant $\CSTu$ depends only on $N$.
\end{lem} 

\begin{proof} 
We follow essentially the proof of~\cite[Lemma~5.6]{CazenaveW}. The main difference is that we use the auxiliary function
\begin{equation} \label{fAl40b1}
v(t)= u(t) - e^{it\Delta }\DI .
\end{equation} 
Observe that by the definition of $\Yset _{ \DI, T, M } $ and~\eqref{fAl25},
$v \in L^\gamma ((0,T),  H^{2,\rho })$, $v_t \in L^\gamma ((0,T), L^{\rho })$, 
and $v(0) = 0$. 
The crux for estimating $v$ is the property $v(0)=0$; and $e^{it\Delta }\DI $ is estimated by Strichartz's estimate.
We deduce from~\eqref{fAl16} (applied with $u=  e^{i \cdot  \Delta } \DI$), 
\eqref{fAl17} and~\eqref{fAl33} that
\begin{equation} \label{fAl39b1}
\begin{split} 
 \| e^{i \cdot  \Delta } \DI \| _{ L^{(\alpha +1)\gamma }((0,T) , L^{(\alpha +1) \rho })}^{\alpha +1}  & \le  \CSob^{\alpha +1} \| \Delta \DI \| _{ L^2 }^\alpha  \| e^{i \cdot  \Delta } \Delta  \DI  \| _{ L^{ \gamma }((0,T) , L^{  \rho  })}  \\ &  \le  \CSob^{\alpha +1} \| \Delta \DI \| _{ L^2 }^\alpha F( \DI, T)  .
\end{split} 
 \end{equation}
 Next,  it follows from~\eqref{fAl16} (applied with $u=  v$) that
\begin{equation} \label{fAl42b1}
 \| v \| _{ L^{(\alpha +1)\gamma } ((0,T), L^{(\alpha +1) \rho }) } ^{\alpha +1}  
  \le     \|  v \| _{ L^{ \gamma }((0,T) , L^{  \Rhostar })}  
 \| v \| _{ L^{\infty  }((0,T) , L^{ \Twostar })}^\alpha .
\end{equation} 
We now estimate $\| v \| _{ L^{\infty  }((0,T) , L^{ \Twostar })}$. 
Since $v(0)= 0$ and $ |\partial _t [ |v|^{\gamma -1} v]|\le \gamma   |v|^{\gamma -1}   |v_t|$, we see that
\begin{equation} \label{fAl44b1}
\begin{split} 
 \| v(t) \| _{ L^{ \Twostar } }^\gamma &= 
  \| \, |v(t)|^{\gamma -1}v(t) \|  _{ L^{ \frac {N} {N-2}} } \\ & = 
  \Bigl\|  \int _0^t \partial _s [ |v(s) |^{\gamma -1} v(s) ]\, ds \Bigr\| _{ L^{ \frac {N} {N-2}} } \\ & 
 \le \gamma  \int _0^T   \|\,  |v|^{\gamma -1}  |v_s|\,\| _{ L^{  \frac {N} {N-2} } } ds.
\end{split} 
\end{equation} 
Since $ \frac {N-2} {N} = \frac {\gamma -1 } { \Rhostar }+ \frac {1} {\rho }$, 
it follows from~\eqref{fAl44b1} and H\"older's inequality in space and in time that
\begin{equation*}
\| v(t) \| _{ L^{ \Twostar } }^\gamma  \le \gamma  \int _0^T  \|v \| _{ L^{ \Rhostar } }^{\gamma -1}   \| v_t \| _{ L^\rho  } dt 
\le \gamma   \| v \|_{ L^\gamma ((0,T), L^{ \Rhostar } ) }^{\gamma -1}  \| v_t \| _{ L^\gamma ((0,T), L^\rho  )},
\end{equation*} 
for all $0<t<T$, so that
\begin{equation}  \label{fAl47b1}
\| v\| _{ L^\infty ((0,T), L^{ \Twostar }) }^\alpha   \le \gamma ^{\frac {2 } {N-2 }}
   \| v \|_{ L^\gamma ((0,T), L^{\Rhostar } ) }^{ \frac {2N} {(N-2)(N-4) }} \| v_t \| _{ L^\gamma ((0,T), L^\rho  )}^{\frac {2 } {N-2 }}  .
\end{equation} 
It follows from~\eqref{fAl42b1} and~\eqref{fAl47b1}  that
\begin{equation} \label{fAl48b1}
 \| v \| _{ L^{(\alpha +1)\gamma } ((0,T), L^{(\alpha +1) \rho }) } ^{\alpha +1}  
 \le \gamma ^{\frac {2 } {N-2 }}
  \|  v \|_{ L^\gamma ((0,T), L^{\Rhostar } ) }^{1+  \frac {2N} {(N-2)(N-4) }} 
   \| v_t \| _{ L^\gamma ((0,T), L^\rho  )}^{\frac {2 } {N-2 }}  .
\end{equation} 
Observe that by~\eqref{fAl40b1} and~\eqref{fAl17}
\begin{equation}  \label{fAl41b3}
\begin{split} 
  \|  v \|_{ L^\gamma ((0,T), L^{\Rhostar } ) } & \le 
   \|  u \|_{ L^\gamma ((0,T), L^{\Rhostar } ) }+  \CSob \|  e^{i \cdot \Delta } \Delta  \DI \|_{ L^\gamma ((0,T), L^{\rho } ) } \\ & \le  \|  u \|_{ L^\gamma ((0,T), L^{\Rhostar } ) } + \CSob F( \DI, T), 
\end{split} 
\end{equation} 
and
\begin{equation}  \label{fAl41b4}
\begin{split} 
  \| v_t \|_{ L^\gamma ((0,T), L^{\rho } ) } & \le 
   \| u_t \|_{ L^\gamma ((0,T), L^{\rho } ) }+  \|  e^{i \cdot \Delta } \Delta  \DI \|_{ L^\gamma ((0,T), L^{\rho } ) } \\ & \le   \| u_t \|_{ L^\gamma ((0,T), L^{\rho } ) } + F( \DI, T) .
\end{split} 
\end{equation} 
Note also that by~\eqref{fAl40b1}
\begin{multline} \label{fAl41b7}
 \| u \| _{ L^{(\alpha +1)\gamma } ((0,T), L^{(\alpha +1) \rho }) } ^{\alpha +1} \\ \le 
 2^\alpha (  \| e^{i \cdot \Delta } \DI  \| _{ L^{(\alpha +1)\gamma } ((0,T), L^{(\alpha +1) \rho }) } ^{\alpha +1} +  \| v \| _{ L^{(\alpha +1)\gamma } ((0,T), L^{(\alpha +1) \rho }) } ^{\alpha +1} ).
\end{multline} 
Estimate~\eqref{fAl49b1} follows from~\eqref{fAl41b7}, \eqref{fAl39b1}, \eqref{fAl48b1}, \eqref{fAl41b3}, \eqref{fAl41b4} and the elementary inequality $(x+y)^{\alpha +1}\le 2^\alpha (x^{\alpha +1} + y^{\alpha +1})$. (Note that the various constants $\alpha $, $\CSob$, $\gamma $ only depend on $N$, so that $\CSTu$ also only depends on $N$.)
\end{proof} 

\begin{lem} \label{eNLem2} 
Given $T,M>0$ and $\DI, \DIb \in \dot H^2 (\R^N ) $, the following properties hold.
\begin{enumerate}[{\rm (i)}] 
\item \label{eNLem2:i}   If $u\in \Yset _{ \DI, T, M } $, then $ |u|^\alpha u \in C([0,T], L^2 (\R^N ) )$. 

\item \label{eNLem2:ii} If $u\in \Yset _{ \DI, T, M } $ and $v\in \Yset _{ \DIb, T, M } $, then
\begin{multline} \label{eNLem2:1} 
 \| \, |u|^\alpha u - |v|^\alpha v \| _{ L^\infty ((0,T), L^2 ) } \le \CSTub 
 \Bigl[   ( \| \DIb \| _{ L^\Twostar }^{\alpha +1} + \| \DI \| _{ L^\Twostar }^{\alpha +1} )  \| \DIb - \DI \| _{ L^\Twostar } \\ + M^{\alpha +1} \| v_t - u_t \| _{ L^\gamma ((0,T), L^\rho )} 
 +  M^{\alpha +1}  \| v -u \| _{ L^\gamma ((0,T), L^\Rhostar )}   \Bigr]^{\frac {\alpha +1} {\alpha +2}},
\end{multline} 
where the constant $\CSTub$ depends only on $N$.
\end{enumerate} 
\end{lem} 

\begin{rem} \label{eRem2} 
Note that  $  \|\,  |u|^\alpha u\| _{ L^2 } =  \| u \| _{ L^\Twostar }^{\alpha +1}$. Therefore, estimate~\eqref{eNLem2:1} with $v=0$, $\DIb=0$ implies that
\begin{equation*} 
 \|  u   \| _{ L^\infty ((0,T), L^\Twostar ) }   \le \CSTub ^{\frac {1} {\alpha +1}}
 \Bigl[     \| \DI \| _{ L^\Twostar }^{\alpha +2}   + M^{\alpha +1}  ( \|  u_t \| _{ L^\gamma ((0,T), L^\rho )} 
 +   \|  u \| _{ L^\gamma ((0,T), L^\Rhostar )}  ) \Bigr]^{\frac {1} {\alpha +2}},
\end{equation*} 
so that
\begin{equation} \label{eNLem2:1b1} 
 \|  u   \| _{ L^\infty ((0,T), L^\Twostar ) }  \le \CSTub ^{\frac {1} {\alpha +1}}
 \Bigl[  \CSob^{\alpha +2}   \|\Delta  \DI \| _{ L^2 }^{\alpha +2}     + (1+\CSob ) M ^{\alpha +2}  \Bigr]^{\frac {1} {\alpha +2}},
\end{equation} 
for all $u\in \Yset _{ \DI, T, M } $.
\end{rem} 

\begin{proof} [Proof of Lemma~$\ref{eNLem2}$]
We first prove Property~\eqref{eNLem2:i}.  
Note that  ${\frac {\Twostar } {\gamma }} = \frac {2(\alpha +1)} {\alpha +2}$, so that
 $ |u|^\alpha u \in C( [0,T] , L^2 (\R^N ) )$ if and only if $ w \in C( [0,T], L^{\frac {\Twostar } {\gamma }} (\R^N ) ) $, where $w=  |u|^{\alpha +1}u$.
Note that 
\begin{equation} \label{fTZ10} 
 \| w \| _{ L^1((0,T), L^{\frac {\Rhostar} {\gamma }}) }=  \| u \|^{\alpha +2} _{  L^\gamma ((0,T), L^\Rhostar) }\le \CSob ^{\alpha +2} M^{\alpha +2}. 
\end{equation} 
Moreover, 
$\frac {\gamma } {\Twostar} = \frac {\alpha +1} {\Rhostar } + \frac {1} {\rho }$,
$ 1= \frac {\alpha +1} {\gamma } + \frac {1} {\gamma }$
and $ |w_t| \le (\alpha +2)  |u|^{\alpha +1}  |u_t|$, 
so that by H\"older's inequality in space and time
\begin{equation} \label{eNLem2:2} 
\begin{split} 
 \|w_t \| _{ L^1((0,T), L^{\frac {\Twostar} {\gamma }} )}&  \le (\alpha +2)  \| u \| _{L^\gamma ((0,T),  L^{\Rhostar} )}^{\alpha +1}  \| u_t \| _{ L^\gamma ((0,T), L^\rho ) } \\ &  \le (\alpha +2) \CSob ^{\alpha +1} M^{\alpha +2}. 
\end{split} 
\end{equation} 
Note that estimates~\eqref{fTZ10} and~\eqref{eNLem2:2}  alone do not imply  $w\in C([0,T], L^{\frac {\Twostar} {\gamma }} (\R^N ) )$. (Let for example $w(t)\equiv w_0$ with $w_0\in  L^{\frac {\Rhostar} {\gamma }}  \setminus L^{\frac {\Twostar} {\gamma }} $.) We use the property $u(0)= \DI$ to complete the proof of~\eqref{eNLem2:i}.  
Let $X= L^{\frac {\Rhostar} {\gamma }} (\R^N ) +L^{\frac {\Twostar} {\gamma }} (\R^N ) $. It follows from~\eqref{fTZ10} and~\eqref{eNLem2:2} that $w\in W^{1,1}((0,T), X) \hookrightarrow C([0,T], X)$. 
In particular, there exists a sequence $t_n\downarrow 0$ such that $w(t_n) \to w(0) $ a.e.  on $\R^N $. 
On the other hand, $u(t) \to \DI $  in $L^\rho  (\R^N ) +L^\Twostar (\R^N )$ as $t \downarrow 0$, by Remark~\ref{eRem1}~\eqref{eRem1:1}. Therefore, by possibly extracting a subsequence, we deduce that $u(t_n) \to \DI$ a.e.  on $\R^N $. Thus $w(t_n) \to |\DI|^{\alpha +1}\DI$ a.e.  on $\R^N $ and we conclude that $w(0)=  |\DI|^{\alpha +1}\DI$. Since $\DI\in L^{2(\alpha +1)} (\R^N ) $ by~\eqref{fAl24}, we conclude that $w(0) \in  L^{\frac {\Twostar} {\gamma }} (\R^N ) $. 
We now write
\begin{equation*} 
w(t)= w(0) + \int _0^t w_t(s)\,ds.
\end{equation*} 
Since $w_t\in L^1((0,T), L^{\frac {\Twostar} {\gamma }} (\R^N ))$ by~\eqref{eNLem2:2}, we see that $w\in  C([0,T], L^{\frac {\Twostar} {\gamma }} (\R^N ))$, which proves Property~\eqref{eNLem2:i}.  

Let now $u,v$ be as in~\eqref{eNLem2:ii}. It follows in particular from~\eqref{eNLem2:i} that 
$ |u|^{\alpha +1} u,  |v|^{\alpha +1} v \in  C( [0,T] , L^{\frac {\Twostar } {\gamma }} (\R^N ) )$.
Applying~\eqref{eElem2v2:1} with $a= \alpha +1$ to both $u $ and $v$, we obtain
\begin{multline*}
\partial _t ( |v|^{\alpha +1} v  -  |u|^{\alpha +1} u)= \frac {\alpha +3} {2}  |v|^{\alpha +1} (v  _t  -u_t) + \frac {\alpha +3} {2} ( |v|^{\alpha +1} -|u|^{\alpha +1})) u_t \\ + \frac {\alpha +1} {2}  |v|^{\alpha -1} v^2 (  \overline{v_t} -  \overline{u}_t) + \frac {\alpha +1} {2} ( |v|^{\alpha -1} v^2-  |u|^{\alpha -1}u^2  ) \overline{u} _t.
\end{multline*} 
Using~\eqref{feRemBSt} with $a =\alpha +1$, we deduce that there exists a constant $C$ depending only on $N$ such that
\begin{multline*}
 |\partial _t ( |v|^{\alpha +1} v  -  |u|^{\alpha +1} u)| \le C ( |v|^{\alpha +1} +  |u|^{\alpha +1})  |v_t -u_t| \\+ C ( |v|^\alpha + |u|^\alpha ) |u_t| \,  |v -u|. 
\end{multline*} 
Applying H\"older's inequality in space and in time, it follows that
\begin{multline*}
 \| \partial _t ( |v|^{\alpha +1} v  -  |u|^{\alpha +1} u) \| _{ L^1((0,T), L^{\frac {\Twostar } {\gamma }} )} \\ \le C (  \| v \| _{ L^\gamma ((0,T), L^\Rhostar )}^{\alpha +1} + \| u \| _{ L^\gamma ((0,T), L^\Rhostar )}^{\alpha +1}) \| v_t - u_t \| _{ L^\gamma ((0,T), L^\rho )} \\
 +C (  \| v \| _{ L^\gamma ((0,T), L^\Rhostar )}^{\alpha } + \| u \| _{ L^\gamma ((0,T), L^\Rhostar )}^{\alpha })  \| v -u \| _{ L^\gamma ((0,T), L^\Rhostar )}  \| u_t \| _{ L^\gamma ((0,T), L^\rho )}.
\end{multline*} 
We deduce by using~\eqref{fAl17} that
\begin{multline*}
 \| \partial _t ( |v|^{\alpha +1} v  -  |u|^{\alpha +1} u) \| _{ L^1((0,T), L^{\frac {\Twostar } {\gamma }})} \\ \le 2C (\CSob M)^{\alpha +1} \| v_t - u_t \| _{ L^\gamma ((0,T), L^\rho )} 
 +2C \CSob^\alpha  M^{\alpha +1}  \| v -u \| _{ L^\gamma ((0,T), L^\Rhostar )}  ,
\end{multline*} 
so that
\begin{multline} \label{eNLem2:8} 
 \|  |v|^{\alpha +1} v  -  |u|^{\alpha +1} u \| _{ L^\infty ((0,T), L^{\frac {\Twostar } {\gamma }} )} 
 \le   \|  \,|\DIb |^{\alpha +1} \DIb  -  |\DI |^{\alpha +1} \DI \| _{L^{\frac {2(\alpha +1)} {\alpha +2} }}\\ + 2C (\CSob M)^{\alpha +1} \| v_t - u_t \| _{ L^\gamma ((0,T), L^\rho )} 
 +2C \CSob^\alpha  M^{\alpha +1}  \| v -u \| _{ L^\gamma ((0,T), L^\Rhostar )}  .
\end{multline} 
Finally, by~\eqref{fSPL1} and H\"older's inequality
\begin{equation}  \label{eNLem2:9} 
  \|  \,|\DIb |^{\alpha +1} \DIb  -  |\DI |^{\alpha +1} \DI \| _{L^{\frac {\Twostar } {\gamma }}}  
  \le (\alpha +2) ( \| \DIb \| _{ L^\Twostar }^{\alpha +1} + \| \DI \| _{ L^\Twostar }^{\alpha +1})  \| \DIb - \DI \| _{ L^\Twostar }  .
\end{equation} 
Since 
\begin{equation*} 
 \| \,  |u|^\alpha u-  |v|^\alpha v\| _{ L^\infty ((0,T), L^2 ) } \le 2^{\frac {\alpha +3} {\alpha +2}}
  \|  |v|^{\alpha +1} v  -  |u|^{\alpha +1} u \| _{ L^\infty ((0,T), L^{\frac {\Twostar } {\gamma }})} ^{\frac {\alpha +1} {\alpha +2}} ,
\end{equation*} 
by~\eqref{fElem1}, estimate~\eqref{eNLem2:1}  follows from~\eqref{eNLem2:8} and~\eqref{eNLem2:9}. 
\end{proof} 

\section{Unconditional uniqueness} \label{sUU}
In this section, we prove unconditional uniqueness in $C([0,T], \dot H^2 (\R^N ) )$ for equation~\eqref{NLS}. 

\begin{prop} \label{eUncUni} 
Let $T>0$, $\DI\in \dot H^2 (\R^N ) $ and suppose $u^1,u^2 \in C([0,T], \dot H^2 (\R^N ) )$ are two solutions of~\eqref{NLSI}. It follows that $u^1=u^2$. 
\end{prop} 

\begin{proof} 
The proof is an obvious adaptation of the proof of Proposition~4.2.5 in~\cite{CLN}.
Note first that, by Sobolev's embedding, 
 $ |u^j|^\alpha u^j  \in C([0,T],  L^{2} (\R^N ) )$, for $j=1,2$.
Therefore, we deduce from equation~\eqref{NLSI} and  Strichartz's estimate~\eqref{fAl26} that 
\begin{equation} \label{fUU1} 
u^j - e^{i\cdot \Delta }\DI \in L^q  ((0,T), L^r  (\R^N ) ), \quad j=1,2,
\end{equation} 
for every admissible pair $(q,r)$. 
Set now 
\begin{equation*} 
S= \sup \{ \tau \in [0,T];\, u^1(t)=u^2 (t)  \text{ for }0\le t\le \tau   \},
\end{equation*} 
so that $0\le S\le T$. Uniqueness follows if we show that $S=T$. 
Assume by contradiction that $S<T$. 
Changing $u^1 (\cdot ), u^2 (\cdot )$ to $u^1 (S+ \cdot ), u^2 (S+ \cdot )$, we are reduced to the case $S=0$, so that
\begin{equation} \label{fUUb2} 
\sup _{ (q,r)\in  \Admis } \| u^1-u^2 \| _{ L^q ((0,\tau ), L^r) } >0\quad  \text{for all}\quad 0<\tau \le T. 
\end{equation} 
On the other hand, it follows from~\eqref{fUU1} that
$u^1- u^2  \in L^q  ((0,T), L^r  (\R^N ) )$
for every admissible pair $(q,r)$.
Moreover, it follows from equation~\eqref{NLSI} (for both $u^1$ and $u^2$) that  
\begin{equation*} 
u^1(t) - u^2(t)= \int _0^t e^{i(t-s) \Delta }[ |u^1(s)|^\alpha u^1(s) - |u^2(s)|^\alpha u^2(s) ]\, ds.
\end{equation*} 
Applying Strichartz's estimate~\eqref{fAl26}, we deduce that
\begin{equation} \label{fUU3} 
\sup _{ (q,r)\in  \Admis } \| u^1-u^2 \| _{ L^q ((0,\tau ), L^r) }\le \CStr  \| \, |u^1|^\alpha u^1- |u^2|^\alpha u^2 \| _{ L^{2} ((0,\tau ), L^{\frac {2N} {N+2}}) },
\end{equation} 
for every $0<\tau \le T$.
On the other hand, it follows from~\eqref{fSPL1}  that
\begin{equation*} 
 |\, |u^1|^\alpha u^1- |u^2|^\alpha u^2| \le f |u^1-u^2|,
\end{equation*} 
where $f= (\alpha +1) ( |u^1|^\alpha +  |u^2|^\alpha )$.
Since $u^1, u^2 \in C([0,T],  L^{\Twostar} (\R^N ) )$, we see that 
\begin{equation} \label{fUU6} 
f\in  C([0,T], L^{\frac {N} {2 }} (\R^N ) ). 
\end{equation} 
Given any $R>0$, we set
\begin{equation*} 
f_R = \min \{  f, R \},\quad f^R=  f- f_R. 
\end{equation*} 
It is not difficult to show (by dominated convergence, using~\eqref{fUU6}) that
\begin{equation*} 
 \| f^R \| _{ L^\infty ((0,T), L^{\frac {N} {2}}) } =: \varepsilon _R \goto  _{ R\to \infty  }0.
\end{equation*} 
Moreover,
\begin{equation*}
\|f_R\| _{ L^\infty ((0,T), L^N ) }\leq R^\frac{1}{2}\|f\|_{L^\infty
((0,T), L^{\frac {N} {2}} ) }^{\frac{1}{2}} =: C(R) <\infty ,
\end{equation*}
for all $R>0$.
Therefore, given any $0<\tau \le T$,
\begin{equation} \label{fUU10} 
 \|f^R  |u^1-u^2| \, \|  _{ L^{2} ((0,\tau ), L^{\frac {2N} {N+2}}) } \le \varepsilon _R  \|u^1 -u^2 \| _{ L^2 ((0,\tau ), L^{\frac {2N} {N-2}} ) },
\end{equation} 
and 
\begin{equation} \label{fUU11} 
\begin{split} 
 \|f_R  |u^1-u^2| \, \|  _{ L^{2} ((0,\tau ), L^{\frac {2N} {N+2}}) } & \le  C(R) \|u^1 -u^2 \| _{ L^2 ((0,\tau ), L^2 ) } \\ & \le C(R) \tau ^{\frac {1} {2}} \|u^1 -u^2 \| _{ L^\infty ((0,\tau ), L^2 ) } .
\end{split} 
\end{equation} 
It follows from~\eqref{fUU3}, \eqref{fUU10} and~\eqref{fUU11}   that
\begin{equation} \label{fUU12} 
\sup _{ (q,r)\in  \Admis } \| u^1-u^2 \| _{ L^q ((0,\tau ), L^r) }\le \CStr 
[ \varepsilon _R + C(R) \tau ^{\frac {1} {2}}
 ] \sup _{ (q,r)\in  \Admis } \| u^1-u^2 \| _{ L^q ((0,\tau ), L^r) }.
\end{equation} 
We first fix $R$ sufficiently large so that $\CStr \varepsilon _R\le \frac {1} {4}$. Then, we choose $0<\tau _0\le T$ sufficiently small so that $\CStr C(R) \tau _0^{\frac {1} {2}}\le \frac {1} {4}$, and we deduce from~\eqref{fUU12} that
\begin{equation*} 
\sup _{ (q,r)\in  \Admis } \| u^1-u^2 \| _{ L^q ((0,\tau _0), L^r) }=0.
\end{equation*} 
This contradicts~\eqref{fUUb2} and proves uniqueness.
\end{proof} 

\section{The local Cauchy problem} \label{sLocal} 

In this section, we prove local existence for the equation~\eqref{NLSI} by a fixed-point argument. More precisely, we have the following result.

\begin{prop} \label{eLocEx} 
Let $M>0$ be sufficiently small so that
\begin{gather} 
  |\lambda |  \CStr  (\alpha +1)  2^\alpha  \CSob ^\alpha  M ^\alpha  \le \frac {1} {2}, \label{fAI50}  \\
  |\lambda | [  \CStr (\alpha +1)  \CSob  ^\alpha
 +  \CSTu  (1+\CSob ^{\alpha +1})  ]  M^{\alpha }   \le \frac {1} {2}, \label{fAI52}
\end{gather} 
where $\CSTu$ is the constant in Lemma~$\ref{eNLem1}$, and $\CStr$ and $\CSob$ are the constants in~\eqref{fAl25}-\eqref{fAl26} and~\eqref{fAl17},  respectively. 
Let $\DI\in \dot H^2 (\R^N ) $,  $T>0$ and suppose further that
\begin{gather}
F( \DI, T) \le \frac {M} {4},   \label{fAI53} \\
(2+   |\lambda |  
 \CSTu \| \Delta \DI \| _{ L^2 }^\alpha)  F( \DI, T) 
  +  |\lambda | \CSTu F( \DI, T)^{\alpha +1}  \le \frac {M} {2}. \label{fAI54}
\end{gather} 
It follows that there exists a solution $u\in  C( [0,T] , \dot H^2 (\R^N ) )  \cap L^\gamma  ((0, T ), \dot H^{2 , \rho }(\R^N ) )$ of~\eqref{NLSI}. 
Moreover, $u\in \Yset _{ \DI ,T,M}$ (given by Definition~$\ref{eSetY}$),  $\Delta u \in L^q ((0,T), L^r (\R^N )) $ and $u_t\in   L^q ((0,T), L^r (\R^N ) ) \cap C ( [0,T] , L^2 (\R^N ) )$ for every  admissible pair $(q,r)$.
If, in addition, $\DI\in L^2 (\R^N ) $, then $u \in   L^q ((0,T), L^r (\R^N ) ) \cap C ( [0,T] , L^2 (\R^N ) )$.
\end{prop} 

\begin{proof} 
We look for a fixed point of the map  $\Phi $ defined by
\begin{equation} \label{fAl19}
\begin{split} 
\Phi (u)(t) & = e^{it\Delta }\DI - i \lambda \int _0^t e^{i(t-s) \Delta }  [|u|^\alpha u ](s) \, ds\\
& = e^{it\Delta }\DI - i \lambda \int _0^t e^{i s \Delta } [ |u|^\alpha u](t-s) \, ds
\end{split} 
\end{equation} 
in the set $\Yset _{ \DI ,T,M}$ of Definition~\ref{eSetY}.  
Note that  $\Phi (u)$ satisfies
\begin{equation} \label{fAl20}
\begin{cases} 
i\Phi _t + \Delta \Phi = \lambda   |u|^\alpha u,\\
\Phi (0)= \DI,
\end{cases} 
\end{equation} 
and that 
\begin{equation} \label{fAl21}
\partial _t \Phi (u) = ie^{it\Delta }[ \Delta \DI -  |\DI|^\alpha \DI ]- i \lambda \int _0^t e^{i(t-s) \Delta } \partial _s[  |u|^\alpha u] (s) \, ds.
\end{equation} 
We first claim that $\Yset _{ \DI ,T,M}$ is nonempty. Indeed, 
 if $ \widetilde{u}  (t) \equiv e^{it\Delta }\DI$, then 
\begin{equation*} 
 \|   \widetilde{u}_t  \| _{ L^\gamma ((0,T), L^{\rho } )}  = \| \Delta  \widetilde{u}  \| _{ L^\gamma ((0,T), L^{\rho } )}  \le F(T, \DI) \le M
\end{equation*} 
by~\eqref{fAI53}.  Since $ \widetilde{u} - e^{i\cdot \Delta }\DI =0$, we see that $ \widetilde{u} \in \Yset _{ \DI ,T,M}$. 
Next, it follows from~\eqref{fAl23} and~\eqref{fAl17} that
\begin{equation} \label{fAl23v2}
\begin{split} 
 \| \partial _t[  |u|^\alpha u]  \| _{ L^{\gamma '}((0,T), L^{\rho '} )} &
  \le (\alpha +1)  \CSob ^\alpha 
  \|\Delta u\| _{ L^\gamma ((0,T), L^{\rho} )}^\alpha  \|u_t\| _{ L^\gamma ((0,T), L^\rho ) } \\
&  \le (\alpha +1)    \CSob  ^\alpha M^{\alpha +1} .
\end{split} 
\end{equation}
Applying~\eqref{fAl21}, \eqref{fAl33}  \eqref{fAl26} and~\eqref{fAl23v2}, we see   that if $u\in \Yset  _{ \DI, T, M } $, then
\begin{equation}  \label{fAl28}
 \| \partial _t \Phi (u) \| _{ L^\gamma ((0,T), L^\rho ) }\le 2 F( \DI, T)
 +   |\lambda |  \CStr (\alpha +1)   \CSob  ^\alpha M^{\alpha +1} \le M,
\end{equation} 
where we used~\eqref{fAI53}  and~\eqref{fAI50} in the last inequality. 
In view of~\eqref{fAl20} we have
\begin{equation} \label{fAl38b1}
 \| \Delta \Phi (u) \| _{ L^\gamma ((0,T), L^\rho ) }\le    \| \partial _t \Phi (u) \| _{ L^\gamma ((0,T), L^\rho ) } +  |\lambda |  \| u \| _{ L^{(\alpha +1)\gamma } ((0,T), L^{(\alpha +1) \rho }) } ^{\alpha +1}.
\end{equation} 
It follows from~\eqref{fAl38b1}, the first inequality in~\eqref{fAl28} and~\eqref{fAl49b1} that
for every $u\in \Yset _{ \DI, T, M } $,
\begin{multline} \label{fAl38b11}
 \| \Delta \Phi (u) \| _{ L^\gamma ((0,T), L^\rho ) }\le (2+   |\lambda |  
 \CSTu \| \Delta \DI \| _{ L^2 }^\alpha)  F( \DI, T) 
  +  |\lambda | \CSTu F( \DI, T)^{\alpha +1}  \\
 +   |\lambda | [  \CStr (\alpha +1)   \CSob  ^\alpha
 +  \CSTu  (1+\CSob ^{\alpha +1})] M^{\alpha +1}   \le M ,
\end{multline} 
where we used~\eqref{fAI54} and~\eqref{fAI52}  in the last inequality. 
 Next, observe that $ |u|^\alpha u\in L^\infty  ((0,T), L^2  (\R^N ) )$ by Lemma~\ref{eNLem2}~\eqref{eNLem2:i}. Therefore, it follows from~\eqref{fAl19} and Stri\-chartz's estimate that  $\Phi (u) - e^{i\cdot \Delta } \DI \in L^\gamma ((0,T),  L^\rho (\R^N ))$. 
Thus we see that   $\Phi : \Yset _{ \DI, T, M }  \mapsto \Yset _{ \DI, T, M } $. 
We next deduce from~\eqref{fSP11} that, given $u,v\in \Yset _{ \DI, T, M }$, 
\begin{equation*} 
 \| \,  |u|^\alpha u-  |v|^\alpha v \| _{ L^{\gamma '}((0,T), L^{\rho '}) } \le (\alpha +1) (2 \CSob M)^\alpha \dist (u,v); 
\end{equation*} 
and so, applying~\eqref{fAl19} and~\eqref{fAl26},
\begin{equation*}
  \| \Phi (u)- \Phi (v)\| _{ L^\gamma ((0,T), L^{\rho } )} 
 \le  |\lambda |  \CStr  (\alpha +1) (2  \CSob  M )^\alpha \dist (u,v) \le \frac {1} {2}  \dist (u,v),
\end{equation*} 
where we used~\eqref{fAl19} in the last inequality. Thus we see that
$ \dist (\Phi (u), \Phi (v)) \le \frac {1} {2}  \dist (u,v)$
 for all $u,v\in \Yset _{ \DI, T, M } $, and it follows from Banach's fixed point theorem that $\Phi $ has a fixed point  $u\in \Yset _{ \DI, T, M } $. In particular, $u$ is a solution of the integral equation~\eqref{NLSI}. 

We now prove the further regularity properties.
We first claim that $u_t, \Delta u\in C( [0,T] , L^2(\R^N ) \cap L^q((0,T), L^r (\R^N ))$ for every admissible pair $(q,r)$.
Indeed, note  that (see~\eqref{fAl21}) 
\begin{equation} \label{fAl21b1}
u_t = ie^{it\Delta }[ \Delta \DI -  |\DI|^\alpha \DI ]- i \lambda \int _0^t e^{i(t-s) \Delta } \partial _s [  |u|^\alpha u] (s) \, ds.
\end{equation}
Since $\Delta \DI -  |\DI|^\alpha \DI \in L^2 (\R^N ) $ and $\partial _t[  |u|^\alpha u] \in L^{\gamma '} ((0,T), L^{\rho '} (\R^N ) )$ by~\eqref{fAl23v2}, it follows from Strichartz's estimates that 
$u_t \in L^q ((0,T), L^r (\R^N ) )$ for every admissible pair $(q,r)$ and $u_t\in C( [0,T]  , L^2 (\R^N ) )$. 
Since $ |u|^\alpha u\in C( [0,T] , L^2 (\R^N ) )$ by Lemma~\ref{eNLem2}~\eqref{eNLem2:i}, it follows from the equation~\eqref{NLS} that  $\Delta u \in C( [0,T] , L^2 (\R^N ) )$. Next, since $\Delta u\in L^\gamma ((0,T), L^\rho (\R^N )) \cap L^\infty ((0,T), L^2 (\R^N ))$, it follows from~\eqref{fFN1v2}  that $ |u|^\alpha u\in L^2 ((0,T), L^{\frac {2N} {N-2}} (\R^N ) )$. Furthermore, $ |u|^\alpha u\in L^\infty ((0,T), L^2 (\R^N ) )$ by  Lemma~\ref{eNLem2}~\eqref{eNLem2:i};
and so, by applying H\"older's inequality, we see that $  |u|^\alpha u\in L^q ((0,T), L^r (\R^N ) )$ for every admissible pair $(q,r)$. Since $\Delta u= -iu_t + \lambda   |u|^\alpha u$, we conclude that $\Delta u  \in L^q ((0,T), L^r (\R^N ) )$.

Finally, suppose further that $\DI\in L^2 (\R^N ) $. Since $ |u|^\alpha u\in L^\infty ((0,T), L^2 (\R^N ) )$ by  Lemma~\ref{eNLem2}~\eqref{eNLem2:i}, we deduce from equation~\eqref{NLSI} and Strichartz estimates~\eqref{fAl25} and~\eqref{fAl26}   that $u \in   L^q ((0,T), L^r (\R^N ) ) \cap C ( [0,T] , L^2 (\R^N ) )$ for every  admissible pair $(q,r)$.
This completes the proof.
\end{proof} 

\section{Continuous dependence} \label{sCDep} 
In this section, we prove continuous dependence on a small time interval.

\begin{prop} \label{eCDBase} 
Let $M>0$ satisfy~\eqref{fAI50}-\eqref{fAI52} and
\begin{gather}
 2 (\alpha +1)  \CStr \CSob ^\alpha  M^\alpha \le \frac {1} {2} ,  \label{fCs02b1} \\
(\alpha +1)  |\lambda | \CStr   \CSob^\alpha  M^\alpha \le \frac {1} {2},  \label{fCs02} \\
2  (\alpha +1) (2M)^{\frac {\alpha } {\alpha +1}} 
  ( |\lambda |  \CSTu )^{\frac {1} {\alpha +1}} \le \frac {1} {2},  \label{fACI53b1}
\end{gather} 
where $\CStr$, $\CSob$ and $\CSTu$ are the constant in~\eqref{fAl25}-\eqref{fAl26}, \eqref{fAl17} and~\eqref{fAl49b1}, respectively. 
Let  $(\DI^n ) _{ n\ge 0 }\subset \dot H^2 (\R^N ) $ and suppose
\begin{equation} \label{fCs01} 
\DI^n \goto  _{ n\to \infty  }\DI^0 \quad  \text{in}\quad \dot H^2 (\R^N ) . 
\end{equation}
Let $T>0$ and suppose further that
\begin{gather}
F( \DI^n , T) \le \frac {M} {4},   \label{fACI53} \\
(2+   |\lambda |  
 \CSTu \| \Delta \DI ^n\| _{ L^2 }^\alpha)  F( \DI^n, T) 
  +  |\lambda | \CSTu F( \DI^n, T)^{\alpha +1}  \le \frac {M} {2}, \label{fACI54}
\end{gather} 
for all $n\ge 1$. 
For every $n\ge 0$, let $u^n \in \Yset _{ \DI^n, T, M }$ be the solution of~\eqref{NLSI} with $\DI$ replaced by $\DI^n$, given by Proposition~$\ref{eLocEx}$. (The assumptions of Proposition~$\ref{eLocEx}$ are satisfied, by~\eqref{fAI50}, \eqref{fAI52}, \eqref{fACI53} and~\eqref{fACI54}.)
It follows that $\Delta u^n \to \Delta u^0$  and $u^n_t \to u^0_t$ in $L^q ((0,T), L^r (\R^N ))$ as $n\to \infty $, for every admissible pair $(q,r)$.
If, in addition,  $(\DI^n) _{ n\ge 0 }\subset L^2 (\R^N ) $ and $\DI^n \to \DI^0$ in $L^2 (\R^N ) $, then $u^n \to u^0$ in $L^q ((0,T), L^r (\R^N ))$.
\end{prop} 

\begin{proof} 
Since $u^n \in \Yset _{ \DI^n, T, M }$, we have
\begin{equation} \label{fCs03} 
  \|u^n _t\| _{ L^\gamma ((0,T), L^\rho ) } \le M, \quad  \| \Delta u^n \| _{ L^\gamma ((0,T), L^\rho ) }\le M,
\end{equation} 
 for all $n\ge 0$. 
We set
\begin{gather} 
\delta _n= \sup  _{ (q,r) \in \Admis} \| u^n_t -u^0_t\| _{ L^q ((0,T), L^r ) },
\label{fCs1} \\
\sigma _n=  \|\Delta ( u^n -u^0) \| _{ L^\gamma ((0,T), L^\rho ) }  , \label{fCs2}
\end{gather} 
where $\Admis$ is defined by~\eqref{fAdmis}. 
We also set 
\begin{equation} \label{fCs06} 
\eta_n = \|\Delta (\DI^n -\DI^0 ) \| _{ L^2 }+  \| \,  |\DI^n|^\alpha \DI^n -  |\DI^0 |^\alpha \DI^0 \| _{ L^2 } 
+  F ( \DI^n -\DI^0 , T),
\end{equation} 
where $F$ is defined by~\eqref{fAl33}. 
It follows from~\eqref{fCs07}, \eqref{fCs01},  and Strichartz's estimate~\eqref{fAl25}    that
\begin{equation}  \label{fCs09} 
\eta _n \goto _{ n\to \infty  }0.
\end{equation} 
We now proceed in five steps.

\Step1 We prove that
\begin{equation} \label{fFSPP11} 
\sup  _{ (q,r) \in \Admis} \| w^n \| _{ L^q ((0,T), L^r )  } \goto  _{ n\to \infty  }0,
\end{equation} 
where $\Admis$ is defined by~\eqref{fAdmis} and 
\begin{equation} \label{fFSPP12} 
w^n = (u^n- e^{i\cdot \Delta }\DI^n )- (u^0- e^{i\cdot \Delta }\DI ^0 ).
\end{equation} 
Indeed, it follows from~\eqref{NLSI} (for $u^0$ and $u^n$) that
\begin{equation} \label{fTZ2} 
w^n (t)=-i \lambda  \int _0^t e^{i(t-s) \Delta }[ |u^n(s)|^\alpha u^n(s)- |u^0(s)|^\alpha u^0(s)]\, ds.
\end{equation} 
It follows from~\eqref{fSPL1} that
\begin{multline} \label{fTZ1} 
 |\,  |u^n(s)|^\alpha u^n(s) -  |u^0 (s)|^\alpha u^0 (s) | 
 \\ \le (\alpha +1) (  |u^n(s)|^\alpha +  |u^0 (s)|^\alpha )  |u^n(s) - u^0 (s)| 
 \le g^n_1 + g^n _2,
\end{multline} 
where
\begin{align*} 
g^n _1 & = (\alpha +1) (  |u^n(s)|^\alpha +  |u^0 (s)|^\alpha )  |w^n(s)| , \\
g^n _2 & = (\alpha +1) (  |u^n(s)|^\alpha +  |u^0 (s)|^\alpha )   |e^{is\Delta }(\DI^n -\DI ^0 )|.
\end{align*} 
Note first that
\begin{equation*} 
\begin{split} 
 \| \, |u^\ell (s)|^\alpha  |w^n(s)|  \| _{ L^{\gamma '}((0,T), L^{\rho '}) } &\le  \| u^\ell  \| _{ L^\gamma ((0, T), L^\Rhostar )}^\alpha  \|w^n \| _{ L^\gamma ((0,T), L^\rho ) } \\ 
 & \le \CSob ^\alpha  \| \Delta u^\ell  \| _{ L^\gamma ((0, T), L^\rho )}^\alpha  \|w^n \| _{ L^\gamma ((0,T), L^\rho ) }  \\ 
 & \le \CSob ^\alpha  M^\alpha  \|w^n \| _{ L^\gamma ((0,T), L^\rho ) },
\end{split} 
\end{equation*} 
for all $\ell, n\ge 0$, by~\eqref{fCs03}.   
Therefore,
\begin{equation} \label{fTZ3} 
 \| g^n _1 \|_{ L^{\gamma '}((0,T), L^{\rho '}) } \le  2 (\alpha +1) \CSob ^\alpha  M^\alpha  \|w^n \| _{ L^\gamma ((0,T), L^\rho ) } .
\end{equation} 
Next, since $\frac {1} {2} = \frac {\alpha } {\Twostar }+ \frac {1} {\Twostar}$, 
it follows from H\"older's estimate and~\eqref{fAl17}  that
\begin{equation*} 
\begin{split} 
\|  \,  |u^\ell (s)|^\alpha     |e^{is\Delta }(\DI^n -\DI ^0)\, | \|  _{ L^2 } & \le  \| u^\ell (s) \| _{ L^\Twostar }^\alpha   \|e^{is\Delta }(\DI^n -\DI ^0)\| _{ L^\Twostar } \\ & \le \CSob  \| u^\ell (s) \| _{ L^\Twostar }^\alpha 
 \| \Delta (\DI^n -\DI ^0) \| _{ L^2 },
\end{split} 
\end{equation*} 
for all $0\le s\le T$ and $\ell, n\ge 0$. 
Applying~\eqref{eNLem2:1b1} and~\eqref{fCs06}, we obtain 
\begin{equation*} 
\|  \,  |u^\ell (s)|^\alpha     |e^{is\Delta }(\DI^n -\DI ^0)\, | \|  _{ L^2 } \le \CSob  
\CSTub ^{\frac {\alpha } {\alpha +1}} 
 \Bigl[  \CSob^{\alpha +2}   \|\Delta  \DI ^\ell \| _{ L^2 }^{\alpha +2}     + (1+\CSob ) M ^{\alpha +2}  \Bigr]^{\frac {\alpha } {\alpha +2}} \eta_n ,
\end{equation*} 
so that
\begin{equation} \label{fTZ4} 
\| g^n _2 \|_{ L^1((0,T), L^2) } \le   2(\alpha +1) T  \CSob  
\CSTub ^{\frac {\alpha } {\alpha +1}} 
 \Bigl[  \CSob^{\alpha +2}   \|\Delta  \DI ^\ell \| _{ L^2 }^{\alpha +2}     + (1+\CSob ) M ^{\alpha +2}  \Bigr]^{\frac {\alpha } {\alpha +2}} \eta_n . 
\end{equation} 
We now set  $g=  |u^n(s)|^\alpha u^n(s) -  |u^0 (s)|^\alpha u^0 (s) $. Since  $ |g| \le  g^n_1 +g^n_2$ by~\eqref{fTZ1}, there exist measurable functions $ \widetilde{g}^n_1 $ and $ \widetilde{g}^n_2 $ such that  $ | \widetilde{g}^n_1|\le g^n_1 $, $ | \widetilde{g}^n_2|\le g^n_2 $ and  $g= \widetilde{g}^n_1 + \widetilde{g}^n_2$.\footnote{For example,  $ \widetilde{g}^n_1 =g$ if $ |g|\le g^n_1 $ and $ \widetilde{g}^n_1 =  g^n_1    |g|^{-1} g $ otherwise.} Therefore, it follows from~\eqref{fTZ2}, \eqref{fTZ3}, \eqref{fTZ4} and Stricharts'z estimate~\eqref{fAl26} that
\begin{multline*} 
 \sup  _{ (q,r) } \| w^n \| _{ L^q ((0,T), L^r )  } \le  2 (\alpha +1)  \CStr \CSob ^\alpha  M^\alpha  \|w_n \| _{ L^\gamma ((0,T), L^\rho ) } \\ 
 + 2 (\alpha +1) T \CStr \CSob  
\CSTub ^{\frac {\alpha } {\alpha +1}} 
 \Bigl[  \CSob^{\alpha +2}   \|\Delta  \DI ^n\| _{ L^2 }^{\alpha +2}     + (1+\CSob ) M ^{\alpha +2}  \Bigr]^{\frac {\alpha } {\alpha +2}} \eta_n. 
\end{multline*} 
Applying~\eqref{fCs02b1}, \eqref{fCs01}  and~\eqref{fCs09}, we conclude that~\eqref{fFSPP11} holds.  

\Step2 We prove that 
\begin{equation} \label{fST1:1} 
\delta _n \goto _{ n\to \infty  }0,
\end{equation} 
where $\delta _n$ is defined by~\eqref{fCs1}. 
Indeed,  we deduce from formula~\eqref{fAl21b1} and Strichartz's estimates~\eqref{fAl25}-\eqref{fAl26} that  
\begin{multline} \label{fCs3}
\delta _n \le \CStr  \|\Delta (\DI^n -\DI ^0) \| _{ L^2 }+ \CStr  \| \,  |\DI^n|^\alpha \DI^n -  |\DI ^0 |^\alpha \DI ^0\| _{ L^2 } \\ +   |\lambda | \CStr
  \| \partial _t (  |u^n|^\alpha u^n -  |u^0 |^\alpha u^0 ) \| _{ L^{\gamma '}((0,T), L^{\rho '} ) }\\
\le \CStr \eta_n +   |\lambda | \CStr
  \| \partial _t (  |u^n|^\alpha u^n -  |u^0|^\alpha u^0) \| _{ L^{\gamma '}((0,T), L^{\rho '} ) }  .
\end{multline} 
We now apply~\eqref{eElem2v2:1} with $a=\alpha $ to both $u$ and $v$ and we obtain
\begin{multline*} 
 |\partial _t ( |v|^\alpha v-  |u|^\alpha u)| \le   (\alpha +1)  |v|^\alpha | v_t -u_t |  \\ 
+ \frac {\alpha +2} {2}  \Bigl[ |\,  |v|^\alpha - |u|^\alpha |+  |\,  |v|^{\alpha -2}v^2 -  |u|^{\alpha -2}u^2|  \,   \Bigr] |u_t|.
\end{multline*} 
Applying~\eqref{feRemBSt} with $a =\alpha $, we deduce that
\begin{equation} \label{fCl9}
 | \partial _t ( |v|^\alpha v-  |u|^\alpha u) | \le (\alpha +1)  |v|^\alpha  | v_t - u_t| +   \CHo F (u,v)   | u_t|, 
\end{equation} 
where 
\begin{equation}  \label{fCl9b1}
F (u,v) =   \begin{cases}
  ( |u|^{\alpha -1} +  |v|^{\alpha -1})   |u-v| & \text{if }\alpha \ge 1, \\
  |u-v|^\alpha &  \text{if }0<\alpha \le 1,
 \end{cases}
\end{equation} 
and the constant $\CHo \ge 1$ depends only on $N$. 
It follows from~\eqref{fCl9}  that
\begin{multline} \label{fCl10v1}
 \| \partial _t ( |u^n|^\alpha u^n-  |u^0 |^\alpha  u^0 ) \| _{ L^{\gamma '} ((0,T), L^{\rho '} )} \le 
 (\alpha +1)  \| u^n \| _{ L^\gamma ((0,T), L^\Rhostar ) }^\alpha   \|  u^n_t -u^0 _t \| _{ L^\gamma ((0,T), L^\rho ) } \\ +  \CHo \| F(u^n, u^0)  | u^0_t |\, \| _{ L^{\gamma '} ((0,T), L^{\rho '} )} \\ 
 \le (\alpha +1) \CSob ^\alpha M^\alpha \delta _n +  \CHo \| F(u^n, u^0)  | u^0_t |\, \| _{ L^{\gamma '} ((0,T), L^{\rho '} )} ,
\end{multline} 
where we used~\eqref{fAl17}, \eqref{fCs03} and~\eqref{fCs1}   in the last inequality.
Applying now~\eqref{fCs3} and~\eqref{fCl10v1}, we see  that
\begin{equation*}
 \delta _n \le \CStr  \eta_n 
 + (\alpha +1)  |\lambda | \CStr   \CSob ^\alpha M^\alpha    \delta _n   
+  |\lambda | \CStr  \CHo \| F(u^n, u^0)  | u^0_t |\, \| _{ L^{\gamma '} ((0,T), L^{\rho '} )} .
\end{equation*} 
which yields, using~\eqref{fCs02}
 \begin{equation}  \label{fCl20v1}
 \delta _n \le 2 \CStr  \eta_n  
+ 2  |\lambda | \CStr  \CHo \| F(u^n, u^0)  | u^0_t |\, \| _{ L^{\gamma '} ((0,T), L^{\rho '} )} .
\end{equation} 
Now, observe that $ | u^0_t|$ is a fixed function of $L^\gamma ((0,T), L^\rho  (\R^N ) )$. Therefore, if we set
\begin{equation} \label{fCl11}
f_R = \min \{  | u^0_t|, R \},\quad f^R=  | u^0_t|- f_R,
\end{equation} 
for $R>0$, then
\begin{equation} \label{fCl12}
 \| f^R \| _{ L^\gamma ((0,T), L^\rho ) }=: \varepsilon _R \goto _{ R\uparrow \infty  }0,
\end{equation} 
and
\begin{equation}  \label{fCl12b1}
 \| f_R \| _{ L^\gamma ((0,T), L^\rho ) }\le   \| u_t \| _{ L^\gamma ((0,T), L^\rho ) }\le  M.
\end{equation} 
Moreover, if $ \widetilde{\rho } \ge  \rho $, then
\begin{equation} \label{fCl13}
 \| f_R (t) \| _{ L^{ \widetilde{\rho } } }^{ \widetilde{\rho }  } \le R^{ \widetilde{\rho } -\rho }  \|  u_t (t)\| _{ L^\rho  }^\rho ,
\end{equation} 
for a.a. $t$, so that
\begin{equation} \label{fCl14}
 \|f_R\| _{ L^\gamma ((0,T), L^{ \widetilde{\rho } }) } \le R^{\frac { \widetilde{\rho }-\rho  } { \widetilde{\rho } }} T^{\frac { \widetilde{\rho }-\rho  } {\gamma  \widetilde{\rho } }}
  \| u_t\| _{ L^\gamma ((0,T), L^{ {\rho } }) } ^{\frac {\rho } { \widetilde{\rho } }}  \le R^{\frac { \widetilde{\rho }-\rho  } { \widetilde{\rho } }} T^{\frac { \widetilde{\rho }-\rho  } {\gamma  \widetilde{\rho } }}
M^{\frac {\rho } { \widetilde{\rho } }}.
\end{equation} 
We now fix $R>0$. 
Writing $ | u_t^0|= f^R+f_R$, we see that
\begin{multline} \label{fCl15}
\| F(u^n, u^0)  | u^0_t |\, \| _{ L^{\gamma '} ((0,T), L^{\rho '} )} \\  \le 
\varepsilon _R  \| F(u^n, u^0)^{\frac {1} {\alpha }} \|_{ L^{\gamma } ((0,T), L^{\Rhostar } )}^\alpha 
 + \| F(u^n, u^0) f_R \| _{ L^{\gamma '} ((0,T), L^{\rho '} )}. 
\end{multline} 
Since $ |F(u^n, u^0) |\le ( |u^n|+ |u^0|)^\alpha $ by~\eqref{fCl9b1}, we deduce from~\eqref{fAl17}  and~\eqref{fCs03} that
\begin{equation*} 
 \| F(u^n, u^0)^{\frac {1} {\alpha }} \|_{ L^{\gamma } ((0,T), L^{\Rhostar } )}^\alpha \le 2^\alpha \CSob ^\alpha M^\alpha ,
\end{equation*} 
and it follows from~\eqref{fCl15} that
\begin{multline} \label{fCl15v1}
\| F(u^n, u^0)  | u^0_t |\, \| _{ L^{\gamma '} ((0,T), L^{\rho '} )} \\  
 \le 2^\alpha \CSob ^\alpha M^\alpha  \varepsilon _R   
 + \| F(u^n -u^0 )  f_R \| _{ L^{\gamma '} ((0,T), L^{\rho '} )}.
\end{multline} 
We now estimate the last term in~\eqref{fCl15v1}, and we first assume 
\begin{equation*}
\alpha \le 1.
\end{equation*} 
We have
\begin{equation} \label{fCl20b2}
 |F(u^n, u^0)| \le  |u^n -u^0 |^\alpha  \le  |w^n|^\alpha +  |e^{it \Delta } (\DI^n - \DI ^0)|^\alpha ,
\end{equation}  
where $w^n$ is defined by~\eqref{fFSPP12}. 
We first estimate, using~\eqref{fCl12b1}, \eqref{fAl17} and~\eqref{fAl25}  
\begin{equation*}  
\begin{split} 
\| \,  |e^{i\cdot  \Delta } (\DI^n - \DI ^0)|^\alpha  f_R \| _{ L^{\gamma '} ((0,T), L^{\rho '} )} &
\le  \| e^{i\cdot  \Delta } (\DI^n - \DI ^0) \| _{ L^\gamma ((0,T), L^\Rhostar )}^\alpha  \| f_R\| _{ L^\gamma ((0,T), L^\rho ) } \\ & \le M \CSob ^\alpha \| e^{i\cdot  \Delta }\Delta  (\DI^n - \DI ^0) \| _{ L^\gamma ((0,T), L^\rho ) }^\alpha \\ &  \le M \CSob ^\alpha \CStr ^\alpha  \| \Delta (\DI^n - \DI ^0)\| _{ L^2 }^\alpha ,
\end{split} 
\end{equation*} 
so that 
\begin{equation} \label{fCl20b3}
\| \,  |e^{i\cdot  \Delta } (\DI^n - \DI ^0)|^\alpha  f_R \| _{ L^{\gamma '} ((0,T), L^{\rho '} )} \le M \CSob ^\alpha \CStr ^\alpha \eta_n ^\alpha .
\end{equation} 
To estimate the contribution of the second term in~\eqref{fCl20b2}, we set
\begin{equation*} 
 \widetilde{\rho }= \frac {2(N-2) (N-4)} {N^2-8N +8} .
\end{equation*} 
It follows that $ \widetilde{\rho } >\rho $ and that $\frac {1} {\gamma '}= \frac {\alpha +1} {\gamma }$, $\frac {1} {\rho '}= \frac {\alpha } {\rho }+ \frac {1} { \widetilde{\rho } }$. Applying H\"older's inequality in space and time,  and~\eqref{fCl14}, we deduce that
\begin{equation} \label{fCl18v2}
\begin{split} 
 \| \, |w^n|^\alpha  f_R \|  _{ L^{\gamma '} ((0,T), L^{\rho '} )} & \le  \| w^n\| _{ L^\gamma ((0,T), L^\rho  )}^\alpha  \|f_R\| _{ L^\gamma ((0,T) ,L^{ \widetilde{\rho } } )} \\
 & \le R^{\frac { \widetilde{\rho }-\rho  } { \widetilde{\rho } }} T^{\frac { \widetilde{\rho }-\rho  } {\gamma  \widetilde{\rho } }} \| w^n\| _{ L^\gamma ((0,T), L^\rho  )}^\alpha  .
\end{split} 
\end{equation}
It now follows from~\eqref{fCl20v1}, \eqref{fCl15v1},  \eqref{fCl20b2},  \eqref{fCl20b3}, and~ \eqref{fCl18v2} that
\begin{multline}  \label{fCl20}
 \delta _n \le
2^{\alpha +1}  |\lambda | \CStr \CHo \CSob^\alpha M^\alpha \varepsilon _R+  
  2 \CStr  \eta_n \\  + 2 |\lambda | M \CSob ^\alpha \CStr ^{\alpha+1} \CHo \eta_n ^\alpha 
+ 2  |\lambda | \CStr  \CHo R^{\frac { \widetilde{\rho }-\rho  } { \widetilde{\rho } }} T^{\frac { \widetilde{\rho }-\rho  } {\gamma  \widetilde{\rho } }}
 \| w^n\| _{ L^\gamma ((0,T), L^\rho  )}^\alpha  .
\end{multline} 
We first let $n\to \infty $ in~\eqref{fCl20}.  
Applying~\eqref{fCs09} and~\eqref{fFSPP11},  we obtain
\begin{equation*} 
\limsup  _{ n\to \infty  } \delta _n \le  2^{\alpha +1}  |\lambda | \CStr \CHo \CSob^\alpha M^\alpha \varepsilon _R.
\end{equation*} 
Since $R>0$ is arbitrary, we may let $R\to \infty $, and~\eqref{fST1:1} follows by using~\eqref{fCl12}.  
We now suppose
\begin{equation*}
\alpha > 1,
\end{equation*} 
and we have
\begin{equation} \label{fCl20v3}
\begin{split} 
 |F(u^n, u^0| & \le ( |u^n|^{\alpha -1}+  |u^0|^{\alpha -1})  |u^n -u^0|   \\ & \le ( |u^n|^{\alpha -1}+  |u^0|^{\alpha -1}) (  |w^n| +  |e^{it \Delta } (\DI^n - \DI ^0)| ). 
\end{split} 
\end{equation} 
We set
\begin{equation*} 
 \widetilde{\rho }= \frac {2 N (N-2) } {N^2-8N +16} .
\end{equation*} 
It follows that $ \widetilde{\rho } >\rho $, and that $\frac {1} {\gamma '}= \frac {\alpha +1} {\gamma }$, $\frac {1} {\rho '}= \frac {\alpha -1} {\Rhostar }+ \frac {1} {\rho }+ \frac {1} { \widetilde{\rho } }$.
We estimate by H\"older's inequality in space and time
\begin{multline*}  
\|  ( |u^n|^{\alpha -1}+  |u^0|^{\alpha -1})   |e^{it \Delta } (\DI^n - \DI ^0)|  f_R \| _{ L^{\gamma '} ((0,T), L^{\rho '} )} \\
\le C (  \|u^n\| _{ L^\gamma ((0,T), L^\Rhostar )}^{\alpha -1} + \|u^0\| _{ L^\gamma ((0,T), L^\Rhostar )}^{\alpha -1}  ) \\
  \| e^{it \Delta } (\DI^n - \DI ^0) \| _{ L^\gamma ((0,T), L^\Rhostar )}  \| f_R\| _{ L^\gamma ((0,T), L^\rho ) } ,
\end{multline*} 
and
\begin{multline*}
 \|  ( |u^n|^{\alpha -1}+  |u^0|^{\alpha -1})  |w^n|  f_R \|  _{ L^{\gamma '} ((0,T), L^{\rho '} )}  
  \\
\le C (  \|u^n\| _{ L^\gamma ((0,T), L^{ {\Rhostar}}  )}^{\alpha -1} + \|u^0\| _{ L^\gamma ((0,T), L^{ {\Rhostar} } )}^{\alpha -1}  ) \\  \| w^n\| _{ L^\gamma ((0,T), L^{{\rho  }} )}  \|f_R\| _{ L^\gamma ((0,T) ,L^{ \widetilde{\rho } } )} ,
\end{multline*}
and it is not difficult to conclude as above that~\eqref{fST1:1} holds.

\Step3 We prove that 
\begin{equation} \label{fST3:1} 
\sigma _n \goto _{ n\to \infty  }0,
\end{equation} 
where $\sigma _n$ is defined by~\eqref{fCs2}. 
Indeed, it follows from the equation~\eqref{NLS} (for $u$ and $u^n$) that
\begin{equation}  \label{fCs13} 
\sigma _n \le  \delta _n +  |\lambda | \,  \|  |u^n|^\alpha u^n-  |u^0|^\alpha u^0 \|_{ L^\gamma ((0,T), L^\rho ) }.
\end{equation} 
Note that by~\eqref{fSPL1} 
\begin{multline} \label{fCs12} 
 \|  |u^n|^\alpha u^n-  |u^0|^\alpha u^0 \|_{ L^\gamma ((0,T), L^\rho ) } \le (\alpha +1) (  \| u^n \| _{ L^{(\alpha +1)\gamma } ((0,T), L^{(\alpha +1) \rho }) } ^\alpha \\ + \| u^0 \| _{ L^{(\alpha +1)\gamma } ((0,T), L^{(\alpha +1) \rho }) } ^\alpha )  \| u^n -u^0 \| _{ L^{(\alpha +1)\gamma } ((0,T), L^{(\alpha +1) \rho }) }.
\end{multline} 
Next, it follows from equation~\eqref{NLS} and~\eqref{fCs03} that 
\begin{equation} \label{fCs10} 
  |\lambda | \,  \| u^n  \| _{ L^{(\alpha +1)\gamma } ((0,T), L^{(\alpha +1) \rho }) }^{\alpha +1}\le 2M,
\end{equation} 
for all $n\ge 0$. 
Moreover, $u^n - u^0 \in \Yset  _{ \DI^n -\DI^0 ,  T, 2M}$. Therefore, \eqref{fAl49b1} yields the estimate
\begin{equation*} 
\| u^n -u^0  \| _{ L^{(\alpha +1)\gamma } ((0,T), L^{(\alpha +1) \rho }) } ^{\alpha +1}  \le 
 \CSTu [ 2 \eta_n^{\alpha +1} + \CSob^{\alpha +1} \sigma _n^{\alpha +1}+ \delta _n^{\alpha +1} ].
\end{equation*} 
Since $( x^{\alpha +1} +y^{\alpha +1} +z^{\alpha +1} )^{\frac {1} {\alpha +1}}\le x+y+z$, we deduce that
\begin{equation}  \label{fCs11} 
\| u^n -u ^0 \| _{ L^{(\alpha +1)\gamma } ((0,T), L^{(\alpha +1) \rho }) }  \le 
 \CSTu ^{\frac {1} {\alpha +1}}[ 2^{\frac {1} {\alpha +1}} \eta_n  + \CSob  \sigma _n + \delta _n ].
\end{equation} 
It follows from~\eqref{fCs12}, \eqref{fCs10} and~\eqref{fCs11} that  
\begin{multline}  \label{fCs14} 
 |\lambda | \, \|  |u^n|^\alpha u^n-  |u^0|^\alpha u^0 \|_{ L^\gamma ((0,T), L^\rho ) } \\
  \le 2  (\alpha +1) (2M)^{\frac {\alpha } {\alpha +1}} 
  ( |\lambda |  \CSTu )^{\frac {1} {\alpha +1}}[ 2^{\frac {1} {\alpha +1}} \eta_n  + \CSob  \sigma _n + \delta _n ].
\end{multline} 
Estimates~\eqref{fCs13} and~\eqref{fCs14} yield
\begin{equation} \label{fCs15} 
\sigma _n \le \delta _n + 2  (\alpha +1) (2M)^{\frac {\alpha } {\alpha +1}} 
  ( |\lambda |  \CSTu )^{\frac {1} {\alpha +1}}[ 2^{\frac {1} {\alpha +1}} \eta_n  + \CSob  \sigma _n + \delta _n ].
\end{equation} 
Applying~\eqref{fACI53b1}, we deduce from~\eqref{fCs15} that 
$\sigma _n \le 3 \delta _n + 2   \eta_n$, 
and~\eqref{fST3:1}  follows from~\eqref{fST1:1} and~\eqref{fCs09}.  

\Step4 We prove that
\begin{equation}  \label{fCs21}
\Delta u^n \goto _{ n\to \infty  }\Delta u^0 \quad  \text{in}\quad  L^q ((0,T), L^r  (\R^N ) ),
\end{equation} 
for every admissible pair $(q,r)$. 
Indeed, note first that by lemma~\ref{eNLem2}~\eqref{eNLem2:ii}, together with Sobolev's inequality~\eqref{fAl17}, \eqref{fCs01}, \eqref{fST1:1} and~\eqref{fST3:1},
\begin{equation}  \label{fCs23}
\| \,  |u^n|^\alpha u^n  - |u^0|^\alpha u^0 \| _{L^\infty ((0,T),  L^2) } \goto _{ n\to \infty  }0.
\end{equation} 
Next, observe that by the equation~\eqref{NLS}, \eqref{fST1:1} and~\eqref{fCs23},    $\Delta u^n$ is bounded, as $n\to \infty $, in $L^\infty ((0,T), L^2 (\R^N ) )$
Therefore, it follows from~\eqref{fFN1v3} and~\eqref{fST3:1} that
\begin{equation}  \label{fCs23b1}
\| \,  |u^n|^\alpha u^n  - |u^0|^\alpha u^0 \| _{L^2((0,T),  L^{\frac {2N} {N-2}}) } \goto _{ n\to \infty  }0.
\end{equation} 
It now follows from~\eqref{fCs23} and~\eqref{fCs23b1} that   
\begin{equation}  \label{fCs23b2}
\| \,  |u^n|^\alpha u^n  - |u^0|^\alpha u^0 \| _{L^q((0,T),  L^r) } \goto _{ n\to \infty  }0,
\end{equation} 
for every admissible pair $(q,r)$. 
Property~\eqref{fCs21} follows from~\eqref{fST1:1}, \eqref{fCs23b2} and the equation~\eqref{NLS}.

\Step5 The case $\DI^n \to \DI^0$ in $L^2 (\R^N ) $. \quad If $(\DI^n) _{ n\ge 0 } \subset  L^2 (\R^N )$ and $\DI^n \to \DI^0$ in $ L^2 (\R^N ) $ as $n\to \infty $, then   by Strichartz's estimate,
\begin{equation*} 
 \| e^{i \cdot \Delta }( \DI^n -\DI ^0)\| _{ L^q ((0,T), L^r )} \goto _{ n\to \infty  }0,
\end{equation*} 
for every admissible pair $(q,r)$. Applying~\eqref{fFSPP11}, we conclude that 
\begin{equation*} 
 \| u^n -u ^0 \| _{ L^q ((0,T), L^r )} \goto _{ n\to \infty  }0,
\end{equation*} 
for every admissible pair $(q,r)$. This completes the proof.
\end{proof} 

\section{Proof of Theorem~$\ref{eComplete}$} \label{sProofMain} 

In this section, we complete the proof of Theorem~\ref{eComplete}.
Note first that uniqueness follows from Proposition~\ref{eUncUni}. 

Fix $M>0$ sufficiently small so that~\eqref{fAI50} and~\eqref{fAI52} are satisfied.  
Given an initial value $\DI\in \dot H^2 (\R^N ) $, it follows from~\eqref{fAl34} that if $T>0$ is sufficiently small, then~\eqref{fAI53} and~\eqref{fAI54} are satisfied. Therefore, it follows from Proposition~\ref{eLocEx} that there exists a solution $u\in C([0,T], \dot H^2 (\R^N ) )$ of~\eqref{NLSI}. We now extend $u$ to a maximal existence interval by the usual procedure. We set
\begin{equation*} 
\Tma= \sup\{\tau >0;\,  \text{there exists a solution }  C([0,\tau ], \dot H^2 (\R^N ) )  \text{ of }\eqref{NLSI}  \},
\end{equation*} 
and it follows from what precedes that $\Tma \ge T>0$. 
By uniqueness, there exists a solution $u\in C([0,\Tma ), \dot H^2 (\R^N ) )$ of~\eqref{NLSI}.  
We now fix $0<S< \Tma$ and 
show that $u\in L^q ((0,S), \dot H^{2,r} (\R^N ) )$ and $u_t \in L^q ((0,S), L^r (\R^N ) )
\cap C([0,S], L^2 (\R^N ) )$ for every admissible pair $(q,r)$. 
Indeed, since $u\in C([0,S], \dot H^2 (\R^N ) )$, we see that $\union _{ 0\le t\le S } \{ u(t) \}$ is a compact subset of $\dot H^2 (\R^N ) $. It then follows from~\eqref{fAl34} that if $T >0$ is sufficiently small, then
\begin{gather}
\sup _{ 0\le t\le S }F( u(t), T ) \le \frac {M} {4},   \label{fAI53v2} \\
\sup _{ 0\le t\le S } (2+   |\lambda |  
 \CSTu \| \Delta \DI \| _{ L^2 }^\alpha)  F( u(t), T ) 
  +  |\lambda | \CSTu F( u(t), T )^{\alpha +1}  \le \frac {M} {2}. \label{fAI54v2}
\end{gather} 
Therefore, we may apply Proposition~\ref{eLocEx} with $\DI$ replaced by $u(t)$ for every $t\in [0,S]$. By uniqueness, we conclude easily that $u$ has the desired regularity properties.
Next, it follows from Lemma~\ref{eNLem2}~\eqref{eNLem2:i} that $ |u|^\alpha u\in  L^\infty   ((0,S), L^2  (\R^N ) )$, so that the further regularity property~\eqref{eComplete:1} follows from Strichartz's estimate~\eqref{fAl26}. If, in addition, $\DI\in L^2 (\R^N ) $, then $e^{i\cdot \Delta }\DI\in C([0,\Tma), L^2 (\R^N ) )$. Therefore, $u \in C([0,\Tma), L^2 (\R^N ) )$, and we conclude that $u \in C([0,\Tma), H^2 (\R^N ) )$.

So far, we have proved the first statements of Theorem~\ref{eComplete}, as well as properties~\eqref{eComplete:2} and~\eqref{eComplete:1}. We now prove property~\eqref{eComplete:3}, 
and we fix $M>0$ sufficiently small so that~\eqref{fAI50} and~\eqref{fAI52} are satisfied.  
We note that by~\eqref{fAl25} and~\eqref{fAl24},   
\begin{equation*} 
F(\DI, \infty ) \le \CStr [  \| \Delta \DI \| _{ L^2 } +   \CSob ^{\alpha +1}  \| \Delta \DI \| _{ L^2 }^{\alpha +1} ].
\end{equation*} 
Therefore, if $ \| \Delta \DI \| _{ L^2 }$ is sufficiently small, then 
\begin{gather}
F( \DI, \infty  ) \le \frac {M} {4},   \label{fAI53v3} \\
(2+   |\lambda |  
 \CSTu \| \Delta \DI \| _{ L^2 }^\alpha)  F( \DI, \infty  ) 
  +  |\lambda | \CSTu F( \DI, \infty  )^{\alpha +1}  \le \frac {M} {2}. \label{fAI54v3}
\end{gather} 
We fix such a $\DI$ and we let $u\in C([0, \Tma), \dot H^2 (\R^N ) )$ be the corresponding solution of~\eqref{NLS}. 
Given any $0< T <\infty $, it follows from~\eqref{fAI53v3}-\eqref{fAI54v3} that we may apply 
Proposition~\ref{eLocEx}. We therefore obtain a solution  of~\eqref{NLS} $u^T \in C([0, T], \dot H^2 (\R^N ) ) \cap \Yset _{ \DI, T, M } $ with $\partial _tu^T\in C([0,T], L^2 (\R^N ) )$. By uniqueness and maximality of $\Tma$, we see that $\Tma >T $ and that $u=u^T$ on $[0,T]$. Since $u^T\in \Yset _{ \DI, T, M } $, we have
$ \| \Delta u\| _{ L^\gamma ((0,T), L^\rho ) }\le M$ and  $ \|u_t\| _{ L^\gamma ((0,T), L^\rho ) } \le M$. 
Therefore, by the blowup alternative we see that $\Tma=\infty $. Thus, we may let $T\to \infty $ and we see that $\Delta u\in  L^\gamma ((0,\infty ), L^\rho (\R^N )) $ and  $ u_t\in  L^\gamma ((0,\infty ), L^\rho (\R^N )) $. 
Next, we deduce from~\eqref{fAl23v2} that $\partial _t[ |u|^\alpha u]\in L^{\gamma '}((0,\infty ), L^{\rho '} (\R^N ))$, so that by~\eqref{fAl21b1} and Strichartz's estimates $u_t \in L^q ((0,\infty ), L^r (\R^N ))$ for every admissible pair $(q,r)$. 
Furthermore, we deduce from Lemma~\ref{eNLem2}~\eqref{eNLem2:ii} that
\begin{equation} \label{fTZ6} 
|u|^\alpha u\in L^\infty ((0,\infty ), L^2 (\R^N ) ),
\end{equation} 
and it follows from~\eqref{NLS} that $\Delta u\in L^\infty ((0,\infty ), L^2 (\R^N ))$. 
Applying~\eqref{fFN1v2}, we deduce that $ |u|^\alpha u\in L^2 ((0,\infty ), L^{\frac {2N} {N-2}} (\R^N ) )$. Interpolating with~\eqref{fTZ6}, we conclude that $  |u|^\alpha u\in L^q ((0,\infty ), L^r (\R^N ) )$ for every admissible pair $(q,r)$. Since $\Delta u= -iu_t + \lambda   |u|^\alpha u$, we see that $\Delta u  \in L^q ((0,\infty ), L^r (\R^N ) )$.

We now prove the blowup alternative~\eqref{eComplete:4}.
Suppose by contradiction that $\Tma <\infty  $ and 
\begin{equation} \label{fFN3}
\| u \| _{ L^\gamma ((0, \Tma), L^\Rhostar  )} <\infty .
\end{equation} 
We first show that 
\begin{gather} 
\| u_t \| _{ L^\gamma ((0, \Tma), L^\rho )} <\infty , \label{fFN4} \\
\| \Delta u \| _{ L^\gamma ((0, \Tma), L^\rho )} <\infty . \label{fFN2}  
\end{gather} 
Fix $\varepsilon >0$ sufficiently small so that
\begin{equation}  \label{fFN3b1}
 (\alpha +1)  |\lambda | \CStr\varepsilon ^\alpha \le \frac {1} {2}.
\end{equation} 
By~\eqref{fFN3}, there  exists $T_\varepsilon \in [0, \Tma)$ such that 
\begin{equation*} 
 \| u \| _{ L^\gamma ((T_\varepsilon , \Tma), L^\Rhostar ) } \le \varepsilon .
\end{equation*} 
Changing $u(\cdot )$ to $u(T_\varepsilon +\cdot )$ and $\DI$ to $u( T_\varepsilon )$, we may assume that $T_\varepsilon =0$, so that
\begin{equation} \label{fFN5} 
 \| u \| _{ L^\gamma ((0 , \Tma), L^\Rhostar ) } \le \varepsilon .
\end{equation} 
We next observe that by~\eqref{fAl21b1}, Strichartz's estimates~\eqref{fAl25}-\eqref{fAl26}, \eqref{fAl24}  and~\eqref{fAl23},
\begin{multline*} 
\| u_t \| _{ L^\gamma ((0, T), L^\rho )}\le \CStr  (\| \Delta \DI\| _{ L^2 } +  \CSob ^{\alpha +1}  \| \Delta \DI \| _{ L^2 }^{\alpha +1}) \\
 + (\alpha +1)  |\lambda | \CStr 
  \|u\| _{ L^\gamma ((0,T), L^{ \Rhostar } )}^\alpha  \|u_t\| _{ L^\gamma ((0,T), L^\rho ) },
\end{multline*} 
for all $0<T<\Tma$. 
Applying~\eqref{fFN5} and~\eqref{fFN3b1}, we deduce that
\begin{equation} \label{fFN6} 
\| u_t \| _{ L^\gamma ((0, T), L^\rho )}\le 2 \CStr  (\| \Delta \DI\| _{ L^2 } +  \CSob ^{\alpha +1}  \| \Delta \DI \| _{ L^2 }^{\alpha +1}) 
\end{equation} 
for all $0<T<\Tma$.  Thus $ \| u_t \| _{ L^\gamma ((0, \Tma), L^\rho )}  <\infty $ and~\eqref{fFN4} holds. We deduce from the equation~\eqref{NLS} that if $0<T<\Tma$, then
\begin{equation} \label{fFN8} 
 \| \Delta u \|_{ L^\gamma ((0 , T), L^\rho ) } \le \| u_t \|_{ L^\gamma ((0 , T), L^\rho ) } +  |\lambda | 
 \| u \| _{ L^{(\alpha +1)\gamma } ((0,T), L^{(\alpha +1) \rho }) } ^{\alpha +1} ,
\end{equation} 
for every $0<T<\Tma$. 
It follows from~\eqref{fFN8},  \eqref{fFN4},  \eqref{fFN3} and~\eqref{fAl49b1} that~\eqref{fFN2} holds. 

Next,  \eqref{fFN4}, \eqref{fFN2}  and~\eqref{fAl23} imply that
$\partial _t[  |u|^\alpha u] \in  L^{\gamma '}((0,\Tma), L^{\rho '} (\R^N ))$, so that by~\eqref{fAl21b1} and Strichartz,  $u_t\in C([0,\Tma], L^2 (\R^N ) )$. Since also $ |u|^\alpha u\in C([0,\Tma], L^2 (\R^N ) )$ by~\eqref{fFN4}, \eqref{fFN2}  and Lemma~\ref{eNLem2}~\eqref{eNLem2:i}, we deduce from  equation~\eqref{NLS} that $\Delta u \in C([0,\Tma], L^2 (\R^N ) )$, so that $u \in C([0,\Tma], \dot H^2 (\R^N ) )$. 
Thus we may apply Proposition~\ref{eLocEx}  and construct a solution $v$ of~\eqref{NLSI} with $\DI$ replaced by $u( \Tma)$, on some time interval $[0,T]$ with $T>0$. Setting 
\begin{equation*} 
 \widetilde{u} (t)=  
 \begin{cases} 
 u(t) & 0\le t\le \Tma ,\\
 v(t- \Tma ) & \Tma \le t\le \Tma +T,
 \end{cases} 
\end{equation*} 
it is not difficult to see that $ \widetilde{u} $ is a solution of~\eqref{NLSI} on $[0, \Tma +T]$, which contradicts the maximality of $\Tma$ and proves the blowup alternative.

It remains to prove the continuous dependence property~\eqref{eComplete:5}. 
This follows from Proposition~\ref{eCDBase} and a standard compactness argument. 
More precisely, let $\DI\in \dot H^2 (\R^N ) $, and let $u$ be the corresponding solution of~\eqref{NLSI}, defined on the maximal interval $[0,\Tma (\DI))$.  
Fix $T< \Tma$, and fix $M>0$ satisfying~\eqref{fAI50}, \eqref{fAI52}, \eqref{fCs02b1}, \eqref{fCs02} and~\eqref{fACI53b1}. Since 
$\union _{ 0\le t\le T } \{ u(t) \}$ is a compact subset of $\dot H^2 (\R^N ) $, it follows from~\eqref{fAl34} that we may fix $\tau  >0$  sufficiently small so that
\begin{gather}
\sup  _{ 0\le t\le T }F( u(t) , \tau ) \le \frac {M} {8},   \label{fACI53v2} \\
\sup  _{ 0\le t\le T } (2+   |\lambda |  
 \CSTu \| \Delta u(t)\| _{ L^2 }^\alpha)  F(u(t) , \tau ) 
  +  |\lambda | \CSTu F( u(t), \tau )^{\alpha +1}  \le \frac {M} {4}. \label{fACI54v2}
\end{gather} 
Let $\ell \ge 1$ be an integer such that $(\ell -1)\tau <T\le \ell \tau $. 
Suppose the sequence $(\DI^n)  _{ n\ge 1 }\subset \dot H^2 (\R^N ) $ satisfies $\DI^n \to \DI$ in $\dot H^2 (\R^N ) $ as $n\to \infty $ and let $u^n$ be the corresponding solutions of~\eqref{NLSI}, with maximal existence time $\Tma (\DI^n)$.  Since $\DI^n \to \DI$, it follows from~\eqref{fACI53v2}-\eqref{fACI54v2} that  there exists $n_1$
\begin{gather*}
F( \DI^n , \tau ) \le \frac {M} {4},    \\
(2+   |\lambda |  
 \CSTu \| \Delta \DI^n\| _{ L^2 }^\alpha)  F(\DI^n , \tau ) 
  +  |\lambda | \CSTu F(\DI^n, \tau )^{\alpha +1}  \le \frac {M} {2},
\end{gather*} 
for all $n\ge n_1$. Therefore, we may apply Proposition~\ref{eCDBase}, and it follows that $\Tma (\DI^n) > \tau $ for $n\ge n_1$ and $\Delta u^n \to \Delta u$  and $u^n_t \to u_t$ in $L^q ((0,\tau ), L^r (\R^N )) $  for every admissible pair $(q,r)$.
If $\tau <T$, we deduce in particular that $u^n( \tau ) \to u(\tau )$ in $ \dot H^2 (\R^N ) )$, so that by~\eqref{fACI53v2}-\eqref{fACI54v2}   there exists $n_2$ such that 
\begin{gather*}
F( u^n(\tau ) , \tau ) \le \frac {M} {4},    \\
(2+   |\lambda |  
 \CSTu \| \Delta u^n(\tau )\| _{ L^2 }^\alpha)  F(u^n(\tau ) , \tau ) 
  +  |\lambda | \CSTu F(u^n(\tau ), \tau )^{\alpha +1}  \le \frac {M} {2},
\end{gather*} 
for all $n\ge n_2$. Applying Proposition~\ref{eCDBase}, we deduce that $\Tma (\DI^n) > 2\tau $ for $n\ge n_2$ and $\Delta u^n \to \Delta u$  and $u^n_t \to u_t$ in $L^q ((0,2 \tau ), L^r (\R^N )) $  for every admissible pair $(q,r)$. We see that we can iterate this argument in order to cover the interval $[0,T]$. 
Finally, if $(\DI^n) _{ n\ge 1 }\subset L^2 (\R^N ) $ and $\DI^n \to \DI$ in $L^2 (\R^N ) $, we obtain  $u^n \to u$ in $L^q ((0,T), L^r (\R^N )) $  for every  admissible pair $(q,r)$ by applying, at each step, the corresponding statement in Proposition~\ref{eCDBase}. 

\appendix

\section{Proof of Lemma~$\ref{eElem2v2}$}

We give the proof of Lemma~\ref{eElem2v2}. It relies on the following property. 

\begin{lem} \label{eApprx1} 
Fix a function $\rho \in C^\infty _\Comp (\R^{N+1} )$, $\rho \ge 0$ with $ \| \rho \| _{ L^1(\R^{N+1} )}= 1$ and, given any $n\ge 1$, set $\rho _n(t,x)= n^{N+1} \rho (nt, nx)$ for $t\in \R$, $x\in \R^N $. 
Let $1\le q,r<\infty $,  $u\in L^q(\R, L^r (\R^N ) )$, and set $u_n = \rho _n \star u$ (where the convolution is on $\R^{N+1} $). It follows that
\begin{equation}  \label{fApprx1} 
 \| u_n \| _{ L^q(\R, L^r) } \le  \| u \| _{ L^q(\R, L^r) },
\end{equation} 
and that  $u_n \to u$ in $ L^q(\R, L^r (\R^N ) )$ as $n\to \infty $.
\end{lem} 

\begin{proof} 
We denote by  $\star _x $ the convolution on $\R^N $. We first prove that, given any $f\in L^1 (\R^{N+1} ) $ and $g\in  L^q(\R, L^r (\R^N ) )$,
\begin{equation} \label{fApprx2} 
 \| f \star g \| _{ L^q(\R, L^r) } \le  \|f\| _{ L^1 (\R^{N+1} ) }  \| g \| _{ L^q(\R, L^r (\R^N )) }. 
\end{equation} 
 Indeed,  
\begin{equation*} 
\begin{split} 
[f\star g ] (t,x ) & = \int _\R   \int  _{ \R^N  } f(t-s, x-y) g(s,y) \,dy \,ds \\
& = \int  _{ \R } [f (t-s, \cdot ) \star_x g(s,\cdot ) ] (x) \,ds.
\end{split} 
\end{equation*} 
Therefore, by Young's inequality for the convolution on $\R^N $,
\begin{equation*} 
 \| [f\star g] (t, \cdot ) \| _{ L^r (\R^N ) }  \le   \int  _{ \R } \| f ( t-s, \cdot )\| _{ L^1(\R^N ) }  \| g (s, \cdot )\| _{ L^r(\R^N ) } \,ds .
\end{equation*} 
We now apply Young's inequality for the convolution is time, and we deduce that
\begin{equation*} 
 \| f\star g \| _{ L^q (\R), L^r(\R^N )) } \le  \| f \|_{ L^1 (\R, L^1(\R^N )) }  \| u \| _{ L^q (\R), L^r(\R^N )) }.
\end{equation*} 
Inequality~\eqref{fApprx2} follows, since $ \| f \|_{ L^1 (\R, L^1(\R^N )) } =  \| f\| _{ L^1 (\R^{N+1}) }$.
Estimate~\eqref{fApprx1} is an immediate consequence of~\eqref{fApprx2}, since $\| \rho _n\| _{ L^1 (\R^{N+1}) }= \| \rho \| _{ L^1 (\R^{N+1}) }=1$. 
The convergence property follows from~\eqref{fApprx1} and a standard density argument, see e.g. the proof of Theorem~4.22 in~\cite{Brezis}. 
Note that this argument uses the density of $C_\Comp (\R^{N+1})$ in $L^q(\R, L^r (\R^N ) )$. 
One can show this as follows. By the classical truncation argument, $C_\Comp (\R, L^r (\R^N ) )$ is dense in $L^q(\R, L^r (\R^N ) )$. Then, given a function $u\in C_\Comp (\R, L^r (\R^N ) )$, the set $\union  _{ t\in \R } \{ u(t) \}$ is a compact subset of $L^r (\R^N ) $. Therefore, by the standard truncation and convolution argument (in $\R^N $), $u$ can be approximated in $L^\infty (\R, L^r (\R^N ) )$ by functions of $C_\Comp (\R^{N+1})$.
\end{proof}

\begin{rem} \label{eApprx2} 
Note that the proof of~\eqref{fApprx2} shows the more general inequality
\begin{equation*} 
 \| f \star g\| _{ L^q (\R, L^r (\R^N ) ) } \le  \| f \| _{ L^{q_1} (\R, L^{r_1} (\R^N ) ) }  \| g\| _{ L^{q_2} (\R, L^{r_2} (\R^N ) ) },
\end{equation*} 
where $1\le q,q_1, q_2, r, r_1, r_2\le \infty $ satisfy $\frac {1} {q}= \frac {1} {q_1}+ \frac {1} {q_2} -1$ and $\frac {1} {r}= \frac {1} {r_1} + \frac {1} {r_2}-1$. 
\end{rem} 

\begin{proof} [Proof of Lemma~$\ref{eElem2v2}$]
For a smooth function $u$, identity~\eqref{eElem2v2:1} follows from straightforward calculations. 
For $u$ as in the statement of Lemma~\ref{eElem2v2}, we extend $u$ and $u_t$ to $\R \times \R^N $ by 
setting
\begin{equation*} 
 \widetilde{u}= 
 \begin{cases} 
 u &  \text{on } (0,T)\times \R^N ,\\ 0 &  \text{elsewhere},
 \end{cases}  
\quad 
 \widetilde{v}= 
 \begin{cases} 
 u_t &  \text{on } (0,T)\times \R^N ,\\ 0 &  \text{elsewhere},
 \end{cases}  
\end{equation*} 
and we consider the sequence $(\rho _n) _{ n\ge 1 }$ given by Lemma~\ref{eApprx1}. 
We set $ \widetilde{u} _n = \rho _n \star  \widetilde{u} $, $ \widetilde{v} _n = \rho _n \star  \widetilde{v} $ and we note that $\rho \in C^\infty _\Comp (\R^{N+1})$, so that $ \widetilde{u} _n, \widetilde{v} _n \in C^\infty (\R^{N+1})$. 
We now fix $0<\varepsilon <\frac {1} {2}$ and we set $K_\varepsilon = (\varepsilon ,1-\varepsilon ) \times \R^N $. We note that for $n\ge n_0$ with $n_0$ sufficiently large, the convolutions giving $ \widetilde{u}_n(x) $ and $ \widetilde{v}_n(x) $ for $x\in K_\varepsilon $ only see the values of $u$ and $u_t$ in $(0,T)\times \R^N $. Thus we see that $\partial _t  \widetilde{u}_n =  \widetilde{v}_n  $ in $K_\varepsilon $ for $n\ge n_0$. Applying formula~\eqref{eElem2v2:1} to $ \widetilde{u} $, we deduce that
\begin{equation}  \label{fApprx3} 
 \partial _t ( | \widetilde{u}_n |^a  \widetilde{u}_n )= \frac {a +2} {2}  | \widetilde{u}_n|^a  \widetilde{v}_n + \frac {a} {2}  | \widetilde{u}_n|^{a -2} \widetilde{u}_n^2   \overline{ \widetilde{v}_n} 
\end{equation} 
in $K_\varepsilon $. We now define $q,r\ge 1$ by
$\frac {a} {q_1}+ \frac {1} {q_2}= \frac {1} {q}$ and $\frac {a} {r_1}+ \frac {1} {r_2}= \frac {1} {r}$.
Applying Lemma~\ref{eApprx1} to both $ \widetilde{u} $ and $ \widetilde{v} $, then H\"older's inequality in space and time, we deduce that $ | \widetilde{u}_n |^a  \widetilde{u}_n \to  |u|^a u$ in 
$L^{\frac {q_1} {a+1}} ((\varepsilon ,T-\varepsilon ), L^{\frac {r_1} {a+1}} (\R^N ) )$
 and $ \frac {a +2} {2}  | \widetilde{u}_n|^a  \widetilde{v}_n + \frac {a} {2}  | \widetilde{u}_n|^{a -2} \widetilde{u}_n^2   \overline{ \widetilde{v}_n} \to  \frac {a +2} {2}  |u|^a  u_t + \frac {a} {2}  | u|^{a -2}u^2   \overline{ u}_t $ in $L^q ((\varepsilon ,T-\varepsilon ), L^r (\R^N ) )$, as $n\to \infty $. 
By possibly extracting a subsequence, we may assume that convergence also holds a.e. in $K_\varepsilon $. Letting $n\to \infty $ in~\eqref{fApprx3} we deduce that~\eqref{eElem2v2:1} holds a.e. in $K_\varepsilon $. Since $0<\varepsilon <\frac {1} {2}$ is arbitrary, we conclude that~\eqref{eElem2v2:1} holds a.e. in $(0,T) \times \R^N $. 
\end{proof}

\end{document}